\documentclass[a4paper,11pt,notitlepage]{amsart}
\usepackage{amsmath,amssymb,mathrsfs,amsthm}

\usepackage{verbatim}
\usepackage[spanish,USenglish]{babel} % espanol, ingles
\usepackage{color}
\usepackage{tikz}
\usepackage{graphicx}

\newtheorem{theorem}{Theorem}
\newtheorem{corollary}[theorem]{Corollary}
\newtheorem{lemma}[theorem]{Lemma}
\newtheorem{example}[theorem]{\it Example}
\newtheorem{proposition}[theorem]{Proposition}
\newtheorem{definition}[theorem]{Definition}

\newtheorem{remark}[theorem]{\it Remark}
\renewcommand{\theequation}{\thesection.\arabic{equation}}

\newcommand{\R}{\mathbb{{R}}}
\newcommand{\Z}{\mathbb{{Z}}}
\newcommand{\N}{\mathbb{{N}}}

\newcommand{\C}{\mathbb{{C}}}

\begin{document}

\title[Ces\`{a}ro sums and algebra homomorphisms of bounded operators]{Ces\`{a}ro sums and algebra homomorphisms of bounded operators}

\author[Abadias]{Luciano Abadias}
\address{Departamento de Matem\'aticas, Instituto Universitario de Matem\'aticas y Aplicaciones, Universidad de Zaragoza, 50009 Zaragoza, Spain.}
%\curraddr{}
\email{labadias@unizar.es}

\author[Lizama]{Carlos Lizama}
\address{Departamento de Matem\'atica y Ciencia de la Computaci\'on, Universidad de Santiago de Chile, Casilla 307-Correo 2, Santiago-Chile, Chile.}
%\curraddr{}
\email{carlos.lizama@usach.cl}

 %   Information for first author
\author[Miana]{Pedro J. Miana}
\address{Departamento de Matem\'aticas, Instituto Universitario de Matem\'aticas y Aplicaciones, Universidad de Zaragoza, 50009 Zaragoza, Spain.}
\email{pjmiana@unizar.es}

\author[Velasco]{M. Pilar Velasco}
\address{Centro Universitario de la Defensa, Instituto Universitario de Matem\'aticas y Aplicaciones, Instituto de Matem\'tica Interdisciplinar, 50090 Zaragoza, Spain.}
\email{velascom@unizar.es}
 %\thanks{The second author is  partially supported by ????}

\thanks{C. Lizama has been partially supported by DICYT, Universidad de Santiago de Chile; Project
CONICYT-PIA ACT1112 Stochastic Analysis Research Network; FONDECYT 1140258 and Ministerio de Educaci\'{o}n CEI Iberus (Spain). L. Abadias,  P. J. Miana and M.P. Velasco have been partially supported  by Project MTM2013-42105-P, DGI-FEDER, of the MCYTS; Project E-64, D.G. Arag\'on, and  Project UZCUD2014-CIE-09, Universidad de Zaragoza.}

% General info
\subjclass[2010]{  47C05, 47A35; 44A55, 65Q20}

\keywords{Fractional sums and differences; algebra homomorphisms; $(C, \alpha)$-bounded operators, Ces\`aro bounded; Abel means}

\begin{abstract}
Let $X$ be a complex Banach space. The connection between algebra homomorphisms defined on subalgebras  of the  Banach algebra $\ell^{1}(\N_0)$ and the algebraic structure  of Ces\`{a}ro sums of a linear operator $T\in \mathcal{B}(X)$ is established. In particular, we show that every $(C, \alpha)$-bounded operator $T$ induces - and is in fact characterized - by such an algebra homomorphism. Our method is based on some sequence kernels,  Weyl fractional difference calculus and convolution Banach algebras that are introduced and deeply examined. To illustrate our results,  improvements to bounds for Abel means, new insights on the $(C,\alpha)$ boundedness of the resolvent operator for temperated $\alpha$-times integrated semigroups,   and  examples of bounded homomorphisms  are given in the last section.
 \end{abstract}

\date{}

\maketitle

\section{Introduction}

Let $X$ be a complex Banach space.   Let $T$ be an operator in  the Banach algebra $\mathcal{B}(X)$ and  denote by $\mathcal{T}$ the discrete semigroup given by $\mathcal{T}(n):=T^{n}$ for $n\in \N_0$.
The Ces\`{a}ro sum of order $\alpha>0$ of $T$, $\{\Delta^{-\alpha} \mathcal{T}(n)\}_{n\in \N_0}\subset \mathcal{B}(X)$, is  defined by
\begin{equation*}
\Delta^{-\alpha} \mathcal{T}(n)x = \displaystyle\sum_{j=0}^n k^{\alpha}(n-j)\mathcal{T}(j) x, \qquad x\in X; \quad n\in \N_0,
\end{equation*}
where
\begin{equation*}\label{eq1}
k^{\alpha}(n) = \frac{\Gamma(\alpha +n)}{\Gamma(\alpha) \Gamma(n+1)}, \qquad n\in  \N_0,
\end{equation*}
is the Ces\`{a}ro kernel. It is well known that Ces\`{a}ro sums is an important concept  that appears in several contexts and ways in the literature. For instance, in Zygmund's book, it appeared in connection with summability of Fourier series \cite[Chapter III, Section 3.11]{Zygmund} and in \cite{Ch-Mu93} in relation with weighted norm inequalities for Jacobi polynomials and series.  %It may also be utilized to derive certain criteria for deciding if a given set of constants are the Fourier constants of an L-integrable function \cite{Moore}.
See also \cite{Gr31} and \cite{Ku39}.
The starting point for our investigation is this definition  of fractional sum of the discrete semigroup $\mathcal{T}$. Certain fractional sums have been used in recent years to develop a theory of fractional differences with interesting applications to boundary value problems and concrete models coming from biological issues, see for example \cite{AtSe10} and \cite{Go12}. Note that  this definition coincides or is connected  with other fractional sums of the discrete semigroup $\mathcal{T}$ on the set $\N_0,$ see \cite[Section 1]{AtEl09} or \cite[Theorem 2.5]{Lizama1}.

%After this observation, it is worthwhile to note that Ces\`{a}ro kernels are connected with the Gr\"unwald-Letnikov discretization of time fractional difference equations based the backward scheme (for $\alpha <0$), see ....

Consider $\phi:\N_0\to \R^+$ a positive weight sequence and
 the Banach algebra  $\ell^{1}_\phi$ (endowed with their natural convolution product).
Suppose $\frac{1}{\phi(\cdot)}\mathcal{T} \in \ell^{\infty}(\mathcal{B}(X)).$ It is well known and easy to show that the semigroup $\mathcal{T}$ induces an algebra homomorphism $\theta:\ell^{1}_\phi \to \mathcal{B}(X)$ defined by
\begin{equation*}\label{homos}
\theta(f)x :=\displaystyle\sum_{n=0}^{\infty}f(n)\mathcal{T}(n)x,\qquad f\in \ell^{1}_\phi, \quad x\in X.
\end{equation*}
Note that in the case that $T$ is a power bounded operator, i.e., $\mathcal{T} \in \ell^{\infty}(\mathcal{B}(X))$, then $\theta:\ell^1 \to \mathcal{B}(X)$. Moreover, this homomorphism is a natural extension of the $Z$-transform, see for example \cite{Elaydi} and references therein.

In general, algebra homomorphisms are  useful tools to treat different interesting aspects of operator theory: Algebra relations,  sharp norm estimations, subordination operators, or ergodic behaviour (as Katznelson-Tzafriri theorems, see \cite{Katznelson}).

As mentioned before, it is remarkable that  Ces\`{a}ro sums have appeared in the literature since some time ago but until now there was not noted their relationship with the theory of  fractional sums  and their algebraic structure. The first main purpose of this paper is to show how  this connection  provide   new insight on properties and characterizations of  Ces\`{a}ro sums, notably concerning their interplay with algebra homomorphisms.

Ces\`{a}ro sums are also a basic  tool to define $(C, \alpha)$-bounded operators,  a natural extension of power-bounded operators.
We recall that a bounded operator $T \in \mathcal{B}(X)$ is $(C,\alpha)$-bounded $(\alpha >0)$ if
\begin{equation*}
\sup_{n} \|\frac{1}{k^{\alpha+1}(n)}\Delta^{-\alpha} \mathcal{T}(n) \| < \infty.
\end{equation*}
See \cite{De00, Su-Ze13}  for  examples and  properties of  $(C,\alpha)$-bounded operators. Note that if $T$ is power bounded,  then $T$ is a $(C, \alpha)$-bounded operator for every $\alpha>0$.
However, there are operators that does not satisfy the power-boundedness condition,  but  $\sup_{n\ge 1}
{1\over n}\|\Delta^{-1} \mathcal{T}(n) \|<\infty,$ as the well-known Assani example shows
$$ T= \left( \begin{array}{rrr}
-1 & 2  \\
 0 & -1  \\
 \end{array} \right),
 $$
see \cite[Section 4.7]{E2}; recently other examples are appeared in  \cite{De00, Ed04,  Su-Ze13, To-Ze, Yo98}.

%Then it is natural to consider Ces\`{a}ro sums  of order $\alpha >0$  of $T$   in order
%to define optimal algebra homomorphisms which extends the expression  given in (\ref{homos}).

The following natural question then arises:  $(Q)$ Can $T$ induce an algebra homomorphism from a  proper subalgebra $\mathcal{A} \subset \ell^{1}$ to $\mathcal{B}(X)$ such that Ces\`{a}ro sums are kernels of this homomorphism?.

The second purpose of this paper is to show that, surprisingly,  the answer to $(Q)$ is positive for every bounded operator such that their Ces\`{a}ro sums are properly bounded (which includes $(C, \alpha)$-bounded operators). More precisely, we  construct appropriate subalgebra  $\tau^{\alpha}(k^{\alpha+1}) \subset \ell^{1}$ and then we  prove that the following assertions are equivalent:
\begin{itemize}
\item[(i)] $T$ is $(C, \alpha)$-bounded operator.
\item[(ii)] There exists a bounded algebra homomorphism $\theta : \tau^{\alpha}(k^{\alpha+1}) \to \mathcal{B}(X)$ such that $\theta(e_1)=T.$
\end{itemize}
In the limit case, the following assertions are equivalent:
\begin{itemize}
\item[(a)] $T$ is power bounded.
\item[(b)] There exists a bounded algebra homomorphism $\theta : \ell^1 \to \mathcal{B}(X)$ such that $\theta(e_1)= T.$
    \item[(c)] For any $0<\alpha<1$, there exist bounded algebra homomorphisms $\theta_\alpha : \tau^{\alpha}(k^{\alpha+1}) \to \mathcal{B}(X)$ such that $\theta_\alpha(e_1)=T$ and
        $\displaystyle{
        \sup_{0<\alpha<1}\Vert \theta_\alpha\Vert <\infty.}
        $
\end{itemize}
This paper is organized as follows: In order to construct a suitable Banach algebra and the corresponding homomorphism,   we introduce in Section 2 the notion of $\alpha$-th fractional Weyl sum as follows:
\begin{equation*}
W^{-\alpha}f(n) = \sum_{j=n}^{\infty} k^{\alpha}(j-n) f(j), \qquad n\in \N_0.
\end{equation*}
see Definition \ref{WeylDifference} below. We state their main algebraic properties in Proposition \ref{WeylSumProp}. Then, we  introduce Banach algebras $\tau^{\alpha}(\phi)$ as the completion of the space of sequences $c_{0,0}$ under the norm
$ q_{\phi}(f):=\displaystyle\sum_{n=0}^{\infty}\phi(n)|W^{\alpha}f(n)|$  (Theorem \ref{th3.1}). The weighted sequences $\phi$ need to verify some summability conditions (Definition \ref{condi}) to  prove that the space $\tau^{\alpha}(\phi)$ is a Banach algebra. It is remarkable that such Banach algebras extends those defined  for $\alpha \in \mathbb{N}_0$  and $\phi=k^{\alpha+1}$ in \cite[Section 4]{Gale}. There they are considered to study subalgebras of analytic functions on the unit disc contained in the Koremblyum and (analytic) Wiener algebra.

Section 3 contains an interesting characterization for the Ces\`{a}ro sum of powers of a given $(C,\alpha)$-bounded operator $T\in \mathcal{B}(X)$ solely in terms of certain functional equation (Theorem \ref{TheoremEcFunc}). The obtained characterization corresponds to an extension of the well known functional equation for the corresponding discrete semigroup $\mathcal{T},$ namely
$$
T^nT^m= T^{n+m}, \quad n,m \in \mathbb{N}_0.
$$
Theorem \ref{homomorphism} gives a complete answer to question $(Q)$ by defining
a bounded algebra homomorphism $\theta:\tau^{\alpha}(\phi)\to \mathcal{B}(X)$ given explicitly by $$\theta(f)x:=\displaystyle\sum_{n=0}^{\infty}W^{\alpha}f(n)\Delta^{-\alpha} \mathcal{T}(n)x,\qquad f\in \tau^{\alpha}(\phi),\quad x\in X.$$
This homomorphism enjoys remarkable properties. The existence of  bounded homomorphisms in these new Banach algebras completely characterizes  the growth of Ces\`{a}ro sums in Corollary \ref{reci}; in particular   bounded homomorphisms from algebras $\tau^{\alpha}(k^{\alpha+1})$ characterizes  $(C,\alpha)$-boundedness (Corollary
 \ref{cor5.7}). Such connection seems to be new in the current literature as well as the functional equation found in the beginning of this  section.

The $Z$-transform technique  may be traced back to De
Moivre around the year 1730. In fact, De Moivre introduced the more
general concept of ``generating functions'' to probability theory. It is interesting compare the $Z$-transform (discrete case) versus Laplace transform (continuous case), see for example \cite[Section 6.7]{Elaydi}. In Section 4, we use the Widder space $C^\infty_W((\omega,\infty),X; \hbox{m})$  where $\hbox{m}$ is Borel measure on $\R_+$, introduced in \cite{cho}, to give a new characterization of summable vector-valued sequences in in terms of $Z$-transform  in Theorem \ref{widder}. We complete the approach given in Section 3 involving the $Z$-transform and  resolvent operators in Theorem \ref{th5.5}.

  Finally, in Section 5  we present several applications, counterexamples and final comments on this paper. A straightforward application is to obtain the Abel means by subordination to the Ces\`{a}ro sums, as Theorem \ref{abels} shows. This point of view allows to improve some previous results given in \cite{LSS}. Some results presented in this paper are inspired in similar ones obtained for $\alpha$-times integrated semigroups, see \cite{GM}. In Section 5.2, we show a natural connection between both operator theories. In Section 5.3, we present some counterexamples of algebra homomorphisms defined from some Banach algebras which cannot be extended to some larger algebras. A future research line, the extension of celebrated Katznelson-Tzafriri  to $(C, \alpha)$-bounded operators, is commented in Section 5.4.

\bigskip

\noindent{\bf Notation.} We denote by $\{e_n\}_{n\in \N_0}$ the set of  canonical sequences given by $e_n(j)=\delta_{n,j}$ where $\delta_{n,j}$ is the known Kronecker delta, i.e., $\delta_{n,j}=1$ is $n=j$ and $0$ in other case.  Let $X$ be a Banach space and $\ell^p(X)$  the set of vector-valued sequences \mbox{$f:\N_0\to X$} such that $\displaystyle\sum_{n=0}^{\infty}\lVert f(n)\rVert^p<\infty,$ for $1\le p<\infty$;  and $c_{0,0}(X)$ the set of vector-valued sequences with finite support. When $X=\C$ we write $\ell^p$ and  $c_{0,0}$ respectively.  It is well known that $\ell^1$ is a Banach algebra with the usual (commutative and associative) convolution product $$(f*g)(n)=\displaystyle\sum_{j=0}^n f(n-j)g(j), \qquad n\in \N_0.$$
Consider $\phi:\N_0\to \R^+$ a positive  sequence, and $\ell^{1}_\phi$  is the Banach spaced  formed by complex sequences $f:\N_0\to \C$ such that $\sum_{n\in \N_0}\phi(n)\vert f(n)\vert <\infty$.  We write $f^{\ast n}=f\ast f^{\ast(n-1)}$ for $n\ge 2$, $f^{\ast 1}=f$ and $f^{\ast 0}=e_0;$ in particular $e_n=e_1^{\ast n}$ for $n\in \N_0$.

Throughout the paper, we use the variable constant convention, in which $C$ denotes a constant
which may not be the same from line to line. The constant is frequently written with subindexes
to emphasize that it depends on some parameters.

\section{Weyl differences  and convolution Banach algebras}
\setcounter{theorem}{0}
\setcounter{equation}{0}

%\begin{definition} For all $f,g\in \ell^1$ we define $$(f\circ g)(n):=\displaystyle\sum_{j=n}^{\infty}f(j-n)g(j)\in \ell^1.$$
%\end{definition}
%It is easy to prove that the above definition is consistent.

%In recent papers, it has been considered fractional extensions of difference operators in $\ell^1(X).$

In this section, we  define certain spaces of sequences that corresponds to an extension in two different directions of those considered in the recent paper  \cite[Definition 4.2]{Gale}. We consider a positive order of regularity in Weyl differences (Definition \ref{WeylDifference}) and different order of growth of Weyl differences (Definition \ref{condi}). These spaces correspond to Banach subalgebras of the space $\ell^1$ and are important to  obtain a further characterization via homomorphisms for Ces\`{a}ro sums in the next section.

 We consider the usual difference operator $\Delta f(n)=f(n+1)-f(n),$ for $n\in\N_{0},$ its powers
$\Delta^{k+1}=\Delta^k\Delta=\Delta\Delta^k,$ for $k\in\N,$ and we write by
$\Delta^0f=f $ and $\Delta^1=\Delta$. It is easy to see that $$\Delta^k f(n)=\displaystyle\sum_{j=0}^k(-1)^{k-j}\binom{k}{j}f(n+j),\qquad n\in\N_{0},$$ see for example \cite[(2.1.1)]{Elaydi} and then $\Delta^m:c_{0,0}\to c_{0,0}$ for $m\in \N_0$. In addition, for $\alpha>0,$ we consider the  well-known scalar sequence $(k^{\alpha}(n))_{n=0}^{\infty}$ defined by $$k^{\alpha}(n):=\frac{\Gamma(n+\alpha)}{\Gamma(\alpha)\Gamma(n+1)}={n+\alpha-1\choose \alpha-1},\qquad n\in\N_{0}.$$
In the classical Zygmund's monographic, the numbers $k^\alpha(n)$ are called as Ces\`{a}ro numbers of order $\alpha$ (\cite[Vol. I, p.77]{Zygmund}) and written by $k^\alpha(n)=A^{\alpha-1}_n$. However  the notation as function $k^\alpha$  will facilitate the understanding of this paper. Kernels $k^{\alpha}$ may equivalently be defined by means of the generating function:
\begin{equation}\label{eq2.1}
\sum_{n=0}^{\infty} k^{\alpha}(n) z^{n} = \frac{1}{(1-z)^{\alpha}}, \quad |z| < 1,\quad \alpha>0,
\end{equation}
 and satisfies the semigroup property, that is, $k^{\alpha}*k^{\beta}=k^{\alpha+\beta}$ for $\alpha, \beta >0$. Furthermore,  the following equality holds: for $\alpha>0$,   \begin{equation}\label{double}
 k^{\alpha}(n)=\frac{n^{\alpha-1}}{\Gamma(\alpha)}(1+O({1\over n})), \qquad n\in \N, \end{equation}
(\cite[Vol. I, p.77 (1.18)]{Zygmund}) and $k^\alpha$ is increasing (as a function of $n$) for $\alpha >1$, decreasing for $1>\alpha >0$ and $k^1(n)=1$ for $n\in \N$ (\cite[Theorem III.1.17]{Zygmund}). It is straightforward to check that  $k^\alpha(n)\le k^\beta(n)$ for $\beta \ge \alpha>0$ and $n\in \N_0$. The Gautschi inequality states that
\begin{equation}\label{gau}
x^{1-s}<{\Gamma(x+1)\over \Gamma(x+s)}<(x+1)^{1-s}, \qquad x\ge 1, \quad 0<s<1,
\end{equation}
(\cite{Gautschi}), which implies that
$$
{(n+1)^{\alpha-1}\over \Gamma(\alpha)}<k^\alpha(n)<{n^{\alpha-1}\over \Gamma(\alpha)}, \qquad n\in \N, \quad 0<\alpha<1.
$$

Note that  when $\alpha=0$ we have
\begin{equation*} \label{rem2.2} k^{0}(n):=\displaystyle\lim_{\alpha\to 0^{+}}k^{\alpha}(n)=e_{0}(n),\qquad n\in \N_0.
\end{equation*}

\begin{lemma} \label{duplicacion} For $\alpha >0$, there exists $C_\alpha>0$ such that
$$
k^\alpha(2n)\le C_\alpha k^\alpha(n), \qquad n\in \N_0.
$$
In particular for $0<\alpha<1$, the following equality holds
$$
k^{\alpha+1}(2n)< 2^\alpha  k^{\alpha+1}(n)\left(1+{1-\alpha\over2(1+\alpha)}\right)^\alpha, \qquad n\in \N_0.
$$
\end{lemma}
\begin{proof} The proof of the first inequality is  straightforward by the inequality (\ref{double}). To show the second inequality, we use the know doubling equality for Gamma function
$$
\Gamma(z)\Gamma(z+{1\over 2})=2^{1-2z}\sqrt{\pi}\Gamma(2z), \qquad \Re z>0,
$$ to obtain  that
\begin{eqnarray*}
k^{\alpha+1}(2n)&=&{\Gamma(\alpha +1+2n)\over \Gamma(\alpha+1)\Gamma(2n+1)}=2^\alpha k^{\alpha +1}(n){\Gamma({\alpha\over 2}+{1\over 2}+n)\Gamma({\alpha\over 2}+{1}+n)\over \Gamma({\alpha}+{1}+n)\Gamma({1\over 2}+n)}, \qquad n\ge 1.
\end{eqnarray*}
We apply the Gautschi inequality (\ref{gau}) to get that
\begin{eqnarray*}
{\Gamma({\alpha\over 2}+{1\over 2}+n)\over \Gamma({1\over 2}+n)}&<&({\alpha\over 2}+{1\over 2}+n)^{\alpha \over 2},\cr
{\Gamma({\alpha\over 2}+{1}+n)\over \Gamma(\alpha+{1}+n)}&<&(\alpha+n)^{-\alpha \over 2},
\end{eqnarray*}
for $0<\alpha<1$ and we conclude that
$$
k^{\alpha+1}(2n)<2^\alpha k^{\alpha +1}(n)\left(1+{1-\alpha\over 2(\alpha +n)}\right)^{\alpha \over 2}\le 2^\alpha k^{\alpha +1}(n)\left(1+{1-\alpha\over 2(1+\alpha)}\right)^{\alpha},
$$
for $n \ge 1$ and $0<\alpha<1$.
\end{proof}

 The Ces\`{a}ro sum of order $\alpha$ of $f$ is defined by $$\Delta^{-\alpha}f(n):=(k^{\alpha}*f)(n)=\displaystyle\sum_{j=0}^n k^{\alpha}(n-j)f(j),\qquad n\in\N_{0}, \alpha>0.$$
 Again we prefer to follow the notation $\Delta^{-\alpha}f(n)$ instead of $S_n^{\alpha-1}$ used in \cite{Zygmund}. Note that  $\Delta^{-\alpha-\beta}f= k^\beta\ast (\Delta^{-\alpha}f)$ and then $\Delta^{-\alpha}\Delta^{-\beta}=\Delta^{-(\alpha+\beta)}=\Delta^{-\beta}\Delta^{-\alpha}$ for $\alpha, \beta >0$,  for more details see again \cite[Vol. I, p.76-77]{Zygmund}. Note also that $\displaystyle\lim_{\alpha\to 0}\Delta^{-\alpha}f(n) =f(n)$ with $\alpha>0$ and $n\in \N_{0}$.

%\begin{example}
%We consider $k^1(n)={\bf 1}(n)\equiv 1$ and $\alpha=m \in \mathbb{N},$ then
%$$
%\Delta^{-m} {\bf 1} = \sum_{j=0}^n k^{m}(j)= (k^{m}*k^1)(n)=k^{m+1}(n)=\frac{(m+n)!}{m! n!},
%$$
%where we have applied the semigroup property of $k^{\alpha}.$ More generally, it is clear that
%$$
%k^{\alpha+1}(n)=\sum_{j=0}^n k^{\alpha}(j) = \frac{1}{n B(\alpha+1,n)}
%$$
%where $B$ denotes the Beta function.
%\end{example}

%Furthermore, it is important to consider the following equivalence, which appears several times throughout the paper: for $\alpha>0$ there exist $0<C_1 <C_2$ such that
%\begin{equation} \label{main}
%C_1 \frac{n^{\alpha-1}}{\Gamma(\alpha)}\leq k^{\alpha}(n) \leq %C_2\frac{n^{\alpha-1}}{\Gamma(\alpha)}, \qquad n\in \N.
%\end{equation}
%In this case we write $k^{\alpha}(n)\sim\frac{n^{\alpha-1}}{\Gamma(\alpha)}.$

 We write $W=-\Delta$, $W^{m}=(-1)^{m}\Delta^m$ for $m\in \N$. The operator $W$ has inverse in $c_{0,0},$ $W^{-1}f(n)=\displaystyle\sum_{j=n}^{\infty}f(j) $ and its iterations are given by the sum
 $$W^{-m}f(n)= \displaystyle\sum_{j=m}^{\infty}\frac{\Gamma(j-n+m)}{\Gamma(j-n+1)\Gamma(m)}f(j)= \sum_{j=n}^{\infty}k^{m}(j-n)f(j),\qquad n\in \N_{0}$$  for each scalar-valued sequence $f$ such that $\displaystyle\sum_{n=0}^{\infty} |f(n)|n^{m} <\infty$, see for example \cite[p.307]{Gale}.
These  facts and the clear connection  with the Weyl fractional calculus motivates the following definition.

\begin{definition}\label{WeylDifference}{\rm Let $f: \mathbb{N}_0 \to X$  and $\alpha>0$ be given. The  Weyl sum of order $\alpha$ of $f$, $W^{-\alpha}f$, is defined by $$W^{-\alpha}f(n):=\displaystyle\sum_{j=n}^{\infty} k^{\alpha}(j-n)f(j),\qquad n\in\N_0,$$
whenever the right hand side makes sense. The  Weyl difference of order $\alpha$ of $f$, $W^\alpha f$, is defined by $$W^{\alpha}f(n):=W^m W^{-(m-\alpha)}f(n)=(-1)^{m}\Delta^m W^{-(m-\alpha)}f(n),\qquad n\in\N_0,$$ for $m=[\alpha]+1,$  whenever the right hand side makes sense. In particular $W^{\alpha}: c_{0,0}\to c_{0,0}$ for $\alpha \in \R$. }
\end{definition}

Observe that if $\alpha\in\N_0,$ the  Weyl difference of order $\alpha$  coincides with the definition  given in \cite[Section 4]{Gale}. Some general properties are shown in the following proposition.

\begin{proposition}\label{WeylSumProp} Let $f\in c_{0,0}(X).$ The following assertions hold:\begin{itemize}
\item[(i)] For $\alpha,\beta>0,$ $W^{-\alpha}W^{-\beta}f=W^{-(\alpha+\beta)}f=W^{-\beta}W^{-\alpha}f.$
\item[(ii)] For $\alpha >0$ and $n\in\N_0$, we have $\displaystyle\lim_{\alpha\to 0^+}W^{-\alpha}f(n)=f(n).$
\item[(iii)] For $\alpha>0,$ $W^{\alpha}W^{-\alpha}f=W^{-\alpha}W^{\alpha}f=f.$
\item[(iv)]For $\alpha >0$ and $n\in\N_0$, we have $\displaystyle\lim_{\alpha\to 0^+}W^{\alpha}f(n)=f(n).$
\item[(v)]\label{SemPro} For all $\alpha,\beta\in\R$ we have $W^{\alpha}W^{\beta}f=W^{\alpha+\beta}f=W^{\beta}W^{\alpha}f.$
\end{itemize}
\end{proposition}
\begin{proof}
(i) It is clear using the Fubini theorem and the semigroup property $k^{\alpha+\beta}=k^{\alpha}*k^{\beta}$ for $\alpha, \beta >0$. (ii) It is sufficient to apply that $f$ has finite support and $\displaystyle\lim_{\alpha\to 0^+}k^{\alpha}(j)=e_{0}(j)$ for $j\in \N_0$.
(iii) We write $m=[\alpha]+1.$ Applying part (i), for $n\in\N_0,$ we have that $$W^{\alpha}W^{-\alpha}f(n)=W^mW^{-(m-\alpha)}W^{-\alpha}f(n)=W^mW^{-m}f(n)=f(n),$$ since $W^{-m}$ is the inverse of $W^m$ in $c_{0,0}(X),$ see \cite[Section 4]{Gale}. On the other hand, \begin{eqnarray*}
    W^{-\alpha}W^{\alpha}f(n)&=&W^{-(\alpha+1-m)}W^{-(m-1)}W^mW^{-(m-\alpha)}f(n)=W^{-(\alpha+1-m)}W^{1}W^{-(m-\alpha)}f(n)\\
    &=&W^{-(\alpha+1-m)}W^{-(m-\alpha)}f(n)-\displaystyle\sum_{j=n}^{\infty}k^{\alpha+1-m}(j-n)W^{-(m-\alpha)}f(j+1)\\
    &=&W^{-1}f(n)-\displaystyle\sum_{j=n+1}^{\infty}k^{\alpha+1-m}(j-n-1)W^{-(m-\alpha)}f(j) \\
    &=&W^{-1}f(n)-W^{-1}f(n+1)=f(n),
    \end{eqnarray*}
    where we use part (i).(iv) It is sufficient to apply that $f$ has finite support and $\displaystyle\lim_{\alpha\to 0^+}k^{1-\alpha}(j)=1$ for $j\in\N_0.$
(v) It is simple to check using the previous results.
\end{proof}

\begin{example}\label{ex2.3}

\begin{itemize} \item[(i)] {\rm Let $\lambda \in \mathbb{C}\backslash\{0\},$ and $p_{\lambda}(n):=\lambda^{-(n+1)}$ for $n\in\N_0.$  An easy computation shows that the sequence $p_{\lambda}$ is a pseudo-resolvent, that is, it satisfies the Hilbert equation
$$
(\mu -\lambda ) (p_{\lambda} * p_{\mu})(n) = p_{\lambda}(n) - p_{\mu}(n), \quad n \in \mathbb{N}_0.
$$
Moreover, the following identity holds
\begin{equation*}\label{inverso}
p_\lambda\ast (\lambda e_0-e_1)=e_0, \qquad \lambda \in \mathbb{C}\backslash\{0\}.
\end{equation*}
We claim that the functions $p_{\lambda}$ are eigenfunctions for the operator $W^{\alpha}$ for $\alpha\in \R$ and $\vert \lambda\vert >1$ : we have, by \eqref{eq2.1}, that
\begin{align*}
W^{-\alpha}p_{\lambda}(n) = \lambda^{-(n+1)}\sum_{j=0}^{\infty}k^{\alpha}(j)\lambda^{-j}=\frac{\lambda^{\alpha}}{(\lambda-1)^{\alpha}}p_{\lambda}(n), \qquad n\in \mathbb{N}_0.
\end{align*}
By Proposition \ref{WeylSumProp} (iii), we obtain that
$$
W^{\alpha}p_{\lambda}= \frac{(\lambda-1)^{\alpha}}{\lambda^{\alpha}} p_{\lambda},\qquad |\lambda|>1.
$$
}
\item[(ii)]{\rm
Let $\alpha\geq 0$ and $n\in\N_0$ be given. We define \begin{equation*}
h_n^{\alpha}(j):=\left\{\begin{array}{ll}
k^{\alpha}(n-j),&j\leq n \\
0,&j>n.
\end{array} \right.
\end{equation*}Functions  $h_n^{\alpha}$ are denoted by $\Gamma^{\alpha-1}_n$ for $\alpha \in \N_0$ in \cite[Section 4]{Gale}. Note that $h_n^\alpha\in c_{0,0}$ for $n\in \N_0$, in fact, $h_n^\alpha\in \hbox{span}\{e_j\,\,\vert\,\, 0\le j\le n\}$,
$h_0^\alpha=e_0$, $h_1^\alpha= \alpha e_0+e_1$, $h_n^{0}:=\lim_{\alpha \to 0^+}h^\alpha_n=e_{n},$ and
\begin{equation}\label{cesarocano}
h_n^{\alpha}(j)=k^{\alpha}(n-j)=\sum_{l=0}^nk^\alpha (n-l)e_l(j)=\sum_{l=0}^nk^\alpha (n-l)e_1^{\ast l}(j), \qquad 0\le j\le n.
\end{equation}

Then for all $\beta\geq 0$ it is easy to check that $W^{-\beta}h_n^{\alpha}= h_{n}^{\alpha+\beta}$, i.e., $$W^{-\beta}h_n^{\alpha}(j)=\displaystyle\sum_{i=j}^{\infty}k^{\beta}(i-j)h_{n}^{\alpha}(i)=h_{n}^{\alpha+\beta}(j),\qquad j\in\N_0.$$
Using  Proposition \ref{WeylSumProp} (iii), we obtain that
$$
W^{\beta}h_{n}^{\alpha}(j)= h_n^{\alpha-\beta}(j),\qquad j\in\N_0,
$$ for  $0\leq\beta\leq\alpha$ and $n\in\N_0.$
}
\end{itemize}
\end{example}

The following remark shows an interesting duality between the operator $\Delta^{-\alpha}$ and $W^{-\alpha}.$ Similar results may be found in \cite[Section 4]{AbdeljaDual} and \cite[Theorem 4.1 and 4.4]{Abdelja2}.
\begin{remark}\label{Duality}{\rm  Let $f,g\in c_{0,0}$, we consider the usual duality product $\langle\, ,\,\rangle$ given by
$$
\langle f,g \rangle: = \sum_{n=0}^{\infty} f(n)g(n).
$$
By Fubini theorem, we get that $ \langle W^{-\alpha} f, g \rangle  = \langle f, \Delta^{-\alpha}g \rangle$ and consequently,
$$ \langle f,g \rangle =  \langle W^{\alpha}f,\Delta^{-\alpha} g \rangle = \langle \Delta^{-\alpha} f, W^{\alpha}g \rangle.  $$}
\end{remark}

The next lemma includes a equality which is a important tool for further developments in this paper. The proof runs parallel to the proof of the integer case  given in \cite[Lemma 4.4]{Gale} and we do not include here.

\begin{lemma}\label{LemmaTech} Let $f,g\in c_{0,0}$ and $\alpha\geq 0,$ then
\begin{eqnarray*}\label{ConvNorm}
W^{\alpha}(f*g)(n)&=&\displaystyle\sum_{j=0}^n W^{\alpha}g(j)\displaystyle\sum_{p=n-j}^n k^{\alpha}(p-n+j)W^{\alpha}f(p) \\
&&-\displaystyle\sum_{j=n+1}^{\infty} W^{\alpha}g(j)\displaystyle\sum_{p=n+1}^{\infty} k^{\alpha}(p-n+j)W^{\alpha}f(p).
\end{eqnarray*}
\end{lemma}

Following definitions are inspired  in \cite[Definition 1.3]{GM}.

\begin{definition}\label{condi} {\rm Let $\alpha >0$.  We say that a positive sequence $\phi$
belongs to the class $\omega_{\alpha, loc}$, if there is a constant $c_{\phi}>0$ such that
\begin{equation}\label{inte}
\left(\sum_{n=0}^j+\sum_{n=p+1}^{j+p}\right)k^{\alpha}(n)\phi(j+p-n)\le c_{\phi} \phi(j)\phi(p), \qquad 1\le j\le p.
\end{equation}

Moreover, we denote by $\omega_\alpha$ the set of
nondecreasing sequences  $\phi\in \omega_{\alpha, loc}$ which are of exponential type  and
satisfy $\displaystyle{\inf_{n\ge 0} (k^{\alpha+1}(n))^{-1}\phi(n)>0
}$}
\end{definition}

Example of sequences in $\omega_\alpha $ are the following ones:
\begin{itemize}
\item[(i)] any nondecreasing sequence $\phi$ satisfying $\max(k^{\alpha+1}(n), \phi(2n))\le M \phi(n)$ for some $M>0$ and for each $n\ge 0$ (in particular $\phi(n)= n^\beta(1+n^\mu)$ with $\beta+\mu \ge \alpha$ and $\beta,\mu \ge 0$ and $\phi(n)=k^\gamma(n)$ with $\gamma\ge \alpha+1$).
    \item[(ii)] $\phi(n)= k^{\alpha+1}(n)\rho(n)$, where $\rho$ is a nondecreasing weight, i.e., $\rho(n+m)\le C\rho(n)\rho(m)$ for $n,m\in \N_0$.
         \item[(iii)] $\phi(n)= k^{\nu+1}(n)e^{\lambda n}$ for $\nu,\lambda >0$.
\end{itemize}
By the equivalence $\displaystyle{k^\alpha(n)\sim {n^{\alpha-1}\over \Gamma(\alpha)}}$, see formula (\ref{double}), equivalent examples may be given in terms of $ n^{\alpha-1}$. The particular case $\phi(n)=k^{\alpha +1}(n)$ will play a fundamental role in this paper, and the condition (\ref{inte}) is improved.

\begin{lemma} For $0<\alpha <1$, the following inequality holds
$$
\left(\sum_{n=0}^j+\sum_{n=p+1}^{j+p}\right)k^{\alpha}(n)k^{\alpha +1}(j+p-n)\le \left(2^{\alpha+1}\left(1+{1-\alpha\over2(1+\alpha)}\right)^\alpha-1\right) k^{\alpha +1}(j)k^{\alpha +1}(p), \quad 1\le j\le p.
$$
\end{lemma}

\begin{proof} For $1\le j\le p$, and $\alpha >0$, we have that
\begin{eqnarray*}
\sum_{n=0}^jk^{\alpha}(n)k^{\alpha +1}(j+p-n)&\le& k^{\alpha +1}(j+p)\sum_{n=0}^jk^{\alpha}(n)=k^{\alpha +1}(j+p)k^{\alpha +1}(j)\cr
\sum_{n=p+1}^{j+p}k^{\alpha}(n)k^{\alpha +1}(j+p-n)&\le&k^{\alpha +1}(j-1)\sum_{n=p+1}^{j+p}k^{\alpha}(n)\le k^{\alpha +1}(j)\left(k^{\alpha +1}(j+p)-k^{\alpha +1}(p)\right).
\end{eqnarray*}
As $k^{\alpha +1}$ is an increasing sequence, we have $k^{\alpha +1}(j+p)\le k^{\alpha +1}(2p)$ for $j\le p$ and we apply the Lemma \ref{duplicacion} to conclude the proof.
\end{proof}

\begin{proposition}\label{pieces} Take $0<\alpha \le \beta$ and $\phi \in \omega_{\alpha, loc}$.
\begin{itemize}
\item[(i)] Then $\omega_{\beta, loc}\subset \omega_{\alpha, loc}$ and $\omega_{\beta}\subset \omega_{\alpha}$.
\item[(ii)] $(k^\alpha \ast \phi)(2n)\le c_\phi \phi^2(n)$ for  $ n\ge 1$.
\item[(iii)] $k^\alpha(n)\le c_\phi \phi(n)\le a^n$ for $n \ge 1$ and some $a>0$.
\item[(iv)] $k^{2\alpha}(2n)\le c\phi^2(n)$ for $n\in \N_0$ and $c>0$.
\item[(v)] $\phi(n+1)\le C\phi(n)$ for some $C>0$ independent of $n\ge 1$.
\item[(vi)] $k^\beta \in \omega_{\alpha, loc}$ if and only if $\beta \ge \alpha+1$.
\end{itemize}

\end{proposition}

\begin{proof} (i) Since $k^\beta(n)\ge k^\alpha (n)$ for  $n \in \N_0$, then  $\omega_{\beta, loc}\subset \omega_{\alpha, loc}$ and $\omega_{\beta}\subset \omega_{\alpha}$ for $\beta \ge \alpha >0$. (ii) It is enough to take $j=p$ in \eqref{inte} to obtain the inequality. (iii) By part (ii), we have that
$$
k^\alpha(n)\phi(n)\le  (k^\alpha \ast \phi)(2n)\le c_\phi \phi^2(n), \qquad n\ge 1,
$$
and we get the first inequality. For $n\ge 1$, we apply the inequality \eqref{inte} $n-1$ times to obtain that
$$
c_\phi\phi(n)= c_\phi k^\alpha(0)\phi (n-1+1)\le c^2_\phi\phi(1)\phi(n-1)\le \left(c_\phi \phi(1)\right)^n.
$$
(iv)  We combine parts (ii), (iii) and the semigroup property of kernels $k^\alpha$ to conclude that $$ c_\phi \phi^2(n)\ge(k^\alpha \ast \phi) (2n)\ge c'(k^\alpha \ast  k^\alpha )(2n)=c'k^{2\alpha}(2n) , \qquad n\in \N_0,$$ for some $c'>0$. (v) Take $j=1$ and $p=n\ge 1$  in \eqref{inte} to get
 $$
 \phi(n+1)=k^\alpha(0)\phi(n+1)\le \sum_{m=0}^1k^\alpha(m)\phi(n+1-m)\le c_\phi \phi(1)\phi(n), \qquad n\ge 1.
 $$
(vi) If $k^\beta \in \omega_{\alpha, loc}$ then we apply \eqref{double} and part (ii) to get
$$(k^\alpha \ast k^\beta)(2n)=k^{\alpha+\beta}(2n)\sim 2^{\alpha+\beta-1}{n^{\alpha +\beta-1}\over \Gamma(\alpha+\beta)}\le c {n^{2(\beta-1)}\over \Gamma^2(\beta)}, \qquad n\ge 1
$$
and we conclude  that $\beta \ge  \alpha+1$. Note that $k^{\alpha+1}\in \omega_{\alpha, loc}$ and then $k^\beta \in \omega_{\alpha, loc}$ for $\beta\ge \alpha +1$ for part (i) and we conclude the proof.
\end{proof}

For $\alpha\geq 0,$ and $\phi \in \omega_{\alpha, loc}$, we define the application $q_{\phi}:c_{0,0}\to [0,\infty)$ given by $$q_{\phi}(f):=\displaystyle\sum_{n=0}^{\infty}\phi(n)|W^{\alpha}f(n)|, \qquad f\in c_{0,0}.$$ Note that for $\alpha=0$ the above application corresponds to the usual norm in $\ell^1_\phi$. In the case $\phi=k^{\alpha +1}$, we write $q_\alpha$ instead of $q_{k^{\alpha+1}}$ and $q_0=\Vert \quad \Vert_1$ for $\alpha \ge 0$.  By \eqref{double}, the norm $q_{\alpha}$ is equivalent to the norm $\widetilde{q_{\alpha}}$ given by $$\widetilde{q_{\alpha}}(f):=|f(0)|+\displaystyle\sum_{n=1}^{\infty}n^{\alpha}|W^{\alpha}f(n)|.$$ This expression was considered for the case $\alpha\in\N_0$ in \cite[Definition 4.2]{Gale}.

 Part of the following result extends \cite[Theorem 4.5]{Gale} and the proof is similar to the proof of \cite[Proposition 1.4]{GM}.  We include the proof to give a complete view of this result.

\begin{theorem}\label{th3.1} Let $\alpha> 0$ and  $\phi \in \omega_{\alpha, loc}$. The application $q_{\phi}$ defines a norm in $c_{0,0}$ and  $$q_{\phi}(f*g)\leq C_{\phi}\,q_{\phi}(f)\,q_{\phi}(g), \qquad f,g\in c_{0,0},$$ with $C_{\phi}>0$ independent of $f$ and $g.$ We denote by $\tau^{\alpha}(\phi)$ the Banach algebra obtained as the completion of $c_{0,0}$ in the norm $q_{\phi}.$ In the case that $\phi \in \omega_{\alpha}$ then

\begin{itemize}

\item[(i)] the operator $\Delta$ is linear and bounded on  $\tau^{\alpha}(\phi)$, $\Delta \in \mathcal{B}(\tau^{\alpha}(\phi))$.
\item[(ii)] $\tau^{\alpha}(\phi)\hookrightarrow\tau^{\alpha}(k^{\alpha+1})\hookrightarrow \ell^1,$ and $\lim_{\alpha \to 0^+}q_\alpha(f)=\Vert f\Vert_1$, for $f\in c_{0,0}$.

    \item[(iii)]  for $0<\alpha<\beta$, $\tau^{\beta}(k^{\beta+1})\hookrightarrow\tau^{\alpha}(k^{\alpha+1})$.
    \item[(iv)] for $0< \alpha <1$,
$$
q_{\alpha}(f*g)\leq \left(2^{\alpha+1}\left(1+{1-\alpha\over2(1+\alpha)}\right)^\alpha-1\right)\,q_{\alpha}(f)\,q_{\alpha}(g), \qquad f, g \in \tau^{\alpha}(k^{\alpha+1}).
$$
\end{itemize}
\end{theorem}
\begin{proof} It is clear that $q_{\alpha}$ is a norm in $c_{0,0}.$ Now, applying Lemma \ref{LemmaTech} we have \begin{eqnarray*}
q_{\phi}(f*g)&\leq&\biggl( \displaystyle\sum_{n=0}^{\infty}\sum_{j=0}^{n}\sum_{p=n-j}^{n}+ \displaystyle\sum_{n=0}^{\infty}\sum_{j=n+1}^{\infty}\sum_{p=n+1}^{\infty}\biggr)\phi(n)k^{\alpha}(p-n+j)|W^{\alpha}g(j)||W^{\alpha}f(p)| \\
&=&\biggl( \displaystyle\sum_{j=0}^{\infty}\sum_{n=j}^{\infty}\sum_{p=n-j}^{n}+ \displaystyle\sum_{j=1}^{\infty}\sum_{n=0}^{j-1}\sum_{p=n+1}^{\infty}\biggr)\phi(n)k^{\alpha}(p-n+j)|W^{\alpha}g(j)||W^{\alpha}f(p)| \\
&=&\biggl( \displaystyle\sum_{j=0}^{\infty}\sum_{p=0}^{\infty}\sum_{n=\max(j,p)}^{p+j}+ \displaystyle\sum_{j=1}^{\infty}\sum_{p=1}^{\infty}\sum_{n=0}^{\min(j,p)-1}\biggr)\phi(n)k^{\alpha}(p-n+j)|W^{\alpha}g(j)||W^{\alpha}f(p)|\\
&\le&\phi(0)|W^{\alpha}g(0)||W^{\alpha}f(0)|+c_\phi \displaystyle\sum_{j=1}^{\infty}\sum_{p=1}^{\infty} \phi(j)\phi(p)|W^{\alpha}g(j)||W^{\alpha}f(p)|\le  C_{\phi}\,q_{\phi}(f)\,q_{\phi}(g)
\end{eqnarray*}
where we use Fubini's Theorem twice and the inequality (\ref{inte}) to show the first inequality.

Now take $\phi \in \omega_\alpha$. (i) It is clear that $\Delta$ is a linear operator and
$$
q_\phi (\Delta(f))=\sum_{n=0}^{\infty}\phi(n)|W^{\alpha}f(n)-W^\alpha f(n+1)|\le q_\phi(f)+\sum_{n=1}^{\infty}\phi(n-1)|W^{\alpha}f(n)\vert\le 2q_\phi(f),
$$
for $ f\in \tau^{\alpha}(\phi)$. (ii) It is clear that $\tau^{\alpha}(\phi)\hookrightarrow\tau^{\alpha}(k^{\alpha+1})\hookrightarrow \ell^1.$ By the Monotone Convergence Theorem and Proposition \ref{WeylSumProp} (ii),
  $$\lim_{\alpha \to 0^+}q_\alpha(f)=\lim_{\alpha \to 0^+}\displaystyle\sum_{n=0}^{\infty}k^{\alpha+1}(n)|W^{\alpha}f(n)|=\displaystyle\sum_{n=0}^{\infty}|f(n)| =\Vert f\Vert_1, \qquad f\in c_{0,0}.$$ (iii) Let $f\in c_{0,0},$ and $0< \alpha <\beta$, then
 \begin{eqnarray*}
q_{\alpha}(f)&=&\displaystyle\sum_{n=0}^{\infty}k^{\alpha+1}(n)|W^{\alpha}f(n)|=\displaystyle\sum_{n=0}^{\infty}k^{\alpha+1}(n)|\displaystyle\sum_{j=n}^{\infty}k^{\beta-\alpha}(j-n)W^{\beta}f(j)| \\
&\leq&\displaystyle\sum_{j=0}^{\infty}|W^{\beta}f(j)|\displaystyle\sum_{n=0}^{j}k^{\beta-\alpha}(j-n)k^{\alpha+1}(n)=\displaystyle\sum_{j=0}^{\infty}k^{\beta+1}(j)|W^{\beta}f(j)|=q_{\beta}(f),
\end{eqnarray*} where we have applied Proposition \ref{WeylSumProp} (v) and the semigroup property of $k^{\alpha}.$ (iv) This inequality follows from Lemma (2.8).
\end{proof}

\begin{example}\label{NormEquiv}{\rm Note that sequence $(h^\alpha_n)_{n\in \N_0} \subset \tau^{\alpha}(\phi)$ with $\phi \in \omega_{\alpha, loc}$: By Example \ref{ex2.3} (ii),  $q_\phi(h^\alpha_n)=\phi(n)$ for $n\in \N_0$. Then the series $\displaystyle\sum_{n=0}^\infty W^\alpha f(n)h_n^\alpha $ converges on $\tau^{\alpha}(\phi)$ for every $f\in \tau^{\alpha}(\phi)$.  By Proposition \ref{pieces} (iii)
$$
\vert f(m)\vert\le  \sum_{n=m}^\infty k^\alpha (n-m)\vert W^\alpha(f)(n)\vert \le c_\phi\sum_{n=m}^\infty\phi(n)\vert W^\alpha(f)(n)\vert \le c_\phi q_\phi(f), \qquad m\in \N_0,
$$
wherever $k^\alpha$  or $\phi$ is non-decreasing functions, i.e., for $\alpha \ge 1$ or $\phi \in \omega_{\alpha}$ for example. And then $
f= \displaystyle\sum_{n=0}^\infty W^\alpha f(n)h_n^\alpha
$ on $\tau^{\alpha}(\phi)$.

Take $\phi \in \omega_{\alpha}$ such that $\phi(n) \le Ca^n$ for $a>1$. Then $p_\lambda \in  \tau^{\alpha}(\phi)$ for $\vert \lambda \vert >a$, where sequences $p_\lambda$ are defined in Example \ref{ex2.3} (i), and
$$
q_\phi(p_\lambda) \le C{\vert \lambda-1\vert^\alpha \over \vert \lambda\vert^{\alpha}(\vert \lambda\vert -a)}, \qquad \vert \lambda \vert >a.
$$
In the particular case $\phi=k^{\gamma}$, then $p_\lambda \in  \tau^{\alpha}(k^{\gamma})$ for $\vert \lambda \vert >1$  and $\gamma\ge \alpha+1,$
 \begin{equation}\label{normass}
q_{k^\gamma}(p_\lambda)= {\vert \lambda-1\vert^\alpha \vert \lambda\vert^{\gamma-\alpha-1}\over (\vert \lambda\vert-1)^{\gamma}}, \qquad \vert \lambda \vert >1,
\end{equation}
where we have applied   Example \ref{ex2.3} (i) and the formula \eqref{eq2.1}}.

\end{example}

\section{Ces\`{a}ro sums and algebra homomorphisms}
\setcounter{theorem}{0}
\setcounter{equation}{0}

In this section we present our main results. The algebra structure of Ces\`{a}ro sums are presented in several ways: functional equation (Theorem \ref{TheoremEcFunc}), algebra homomorphism (Theorem \ref{homomorphism}) and resolvent operators (Theorem \ref{th5.5}). Note that these approach in fact  characterizes  the growth of Ces\`{a}ro sums, as Corollary \ref{reci} and Corollary \ref{cor5.7} for $(C, \alpha)$-bounded operators show.
We recall the following definition.

\begin{definition}\label{cesarosum}{\rm
Given a bounded operator $T \in \mathcal{B}(X)$, the Ces\`{a}ro sum of order $\alpha>0$ of $T$, $(\Delta^{-\alpha} \mathcal{T}(n))_{n\ge 0}\subset \mathcal{B}(X)$, is  defined by
\begin{equation*}\label{cesaro}
\Delta^{-\alpha} \mathcal{T}(n)x:=(k^\alpha \ast  \mathcal{T})(n)x = \displaystyle\sum_{j=0}^n k^{\alpha}(n-j)T^j x, \qquad x\in X; \quad n\in \N_0. \end{equation*}
Note that we keep the notation  $ \mathcal{T}(n)=T^n$ for $n\in \N_0$.}
\end{definition}

\begin{example}\label{canoni} {\rm The canonical example of a family of Ces\`{a}ro sum of order $\alpha$ in Banach algebras $\tau^{\alpha}(\phi)$ (in particular in $\ell^1$) is the family $\{h^\alpha_n\}_{n\in \N_0}$ given in Example \ref{ex2.3}(ii). Note that  $(h^\alpha_n)_{n\in \N_0} \subset \tau^{\alpha}(\phi)$ with $\phi \in \omega_{\alpha, loc}$, see Example \ref{NormEquiv}.  We write $ \mathcal{E}(n)=e^{\ast n}_1$ to get
$h_n^\alpha =\Delta^{-\alpha} \mathcal{E}(n)$ for $n\in \N_0$ by equation (\ref{cesarocano}).}
\end{example}

The following theorem characterizes sequences of operators which are Ces\`{a}ro sums  of some order $\alpha>0$  and a fixed operator $T$.

\begin{theorem}\label{TheoremEcFunc} Let $\alpha >0$ and $T, (T_n)_{n\in \N_0}\subset \mathcal{B}(X)$. Then the following assertions are equivalent.
\begin{itemize}
\item[(i)] $T_n = \Delta^{-\alpha}\mathcal{T}(n)$ for $n\in \N_0$.
\item[(ii)]
$T_0 = I$ and the  following functional equation holds: \begin{equation}\label{eq4.2}
 T_nT_m=\displaystyle\sum_{u=m}^{n+m}k^{\alpha}(n+m-u)T_u-\displaystyle\sum_{u=0}^{n-1}k^{\alpha}(n+m-u)T_u \qquad n\geq 1,\, m\in\N_0.
 \end{equation}
\end{itemize}
\end{theorem}
\begin{proof}
Assume (i). It is clear
$T_0 = I$ and we claim the identity \eqref{eq4.2}.
Take $n\geq 1,\, m\geq 0,$ then \begin{eqnarray*}
    T_{n}T_m&=&\displaystyle\sum_{j=0}^n \displaystyle\sum_{i=0}^mk^{\alpha}(n-j)k^{\alpha}(m-i)T^{j+i}
    =\displaystyle\sum_{j=0}^n \displaystyle\sum_{u=j}^{m+j}k^{\alpha}(n-j)k^{\alpha}(m+j-u)T^{u} \\
    &=&\displaystyle\sum_{j=0}^n \displaystyle\sum_{u=0}^{m+j}k^{\alpha}(n-j)k^{\alpha}(m+j-u)T^{u}
-\displaystyle\sum_{j=1}^n \displaystyle\sum_{u=0}^{j-1}k^{\alpha}(n-j)k^{\alpha}(m+j-u)T^{u} \\
    &=&\displaystyle\sum_{j=0}^n k^{\alpha}(n-j)T_{m+j} -\displaystyle\sum_{j=1}^n \displaystyle\sum_{u=0}^{j-1}k^{\alpha}(n-j)k^{\alpha}(m+j-u)T^{u}.
    \end{eqnarray*}
    Observe that \begin{eqnarray*}
    &\quad&\displaystyle\sum_{j=1}^n\displaystyle\sum_{u=0}^{j-1}k^{\alpha}(n-j)k^{\alpha}(m+j-u)T^{u}=\displaystyle\sum_{u=0}^{n-1}\displaystyle\sum_{j=u+1}^n k^{\alpha}(n-j)k^{\alpha}(m+j-u)T^{u} \\
    &\quad&=\displaystyle\sum_{u=0}^{n-1}\displaystyle\sum_{l=u}^{n-1} k^{\alpha}(l-u)k^{\alpha}(m+n-l)T^{u}= \displaystyle\sum_{l=0}^{n-1}k^{\alpha}(m+n-l)\displaystyle\sum_{u=0}^{l} k^{\alpha}(l-u)T^{u} \\
    &\quad&=\displaystyle\sum_{l=0}^{n-1}k^{\alpha}(m+n-l)T_l. \\
    \end{eqnarray*}
    and the equality  \eqref{eq4.2} is proven.    Conversely, assume (ii). Define $T:= T_1 -\alpha I$ and
\begin{equation*}
S_n :=\sum_{j=0}^n k^{\alpha}(n-j)T^j, \quad n \in \mathbb{N}_0.
\end{equation*}
 It is clear that $S_0 = I = T_0,$ and
$
S_1= \alpha I + T = T_1 .
$
Inductively, we suppose that $S_n=T_n.$ Then using that $S_n$ satisfies \eqref{eq4.2}, we have that $$S_{n+1}+k^{\alpha}(1)S_n-k^{\alpha}(n+1)I=S_nS_1=T_nT_1=T_{n+1}+k^{\alpha}(1)S_n-k^{\alpha}(n+1)I.$$  Then we conclude that $T_{n+1}=S_{n+1},$ and consequently $T_{n}= \Delta^{-\alpha} \mathcal{T}(n)$ for all $n \in \mathbb{N}_0.$
\end{proof}

\begin{remark}\label{generador}{\rm Given  $\{ T_n \}_{n\in\N_0}\subset  \mathcal{B}(X)$ a sequence of bounded operators which verify the equality (\ref{eq4.2}).  Then the operator defined by $T:=T_1-\alpha I$ is called the generator of $\{ T_n \}_{n\in\N_0}.$  By Theorem \ref{TheoremEcFunc}, $T_{n}= \Delta^{-\alpha} \mathcal{T}(n)$ where $\mathcal{T}(n)=T^n$ for $n\in \N_0$. In particular, note that $\{h_n^{\alpha}\}_{n\in\N_0}$ satisfies \eqref{eq4.2} in $\tau^{\alpha}(\phi),$ see Example \ref{canoni}, and the generator is the element $e_1$. }
\end{remark}

The following is one of the main results of this paper.

\begin{theorem}\label{homomorphism} Let $\alpha> 0$ and  $T \in \mathcal{B}(X)$ such that  $\Vert \Delta^{-\alpha} \mathcal{T}(n)\Vert \le C\phi(n)$  for $n\in \N_0$ with $\phi\in \omega_{\alpha, loc}$ and $C>0$. Then there exists a bounded algebra homomorphism $\theta:\tau^{\alpha}(\phi)\to \mathcal{B}(X)$ given by $$\theta(f)x:=\displaystyle\sum_{n=0}^{\infty}W^{\alpha}f(n)\Delta^{-\alpha} \mathcal{T}(n)x,\qquad x\in X,\quad f\in \tau^{\alpha}(\phi).$$

 Furthermore,  the following identities  hold.

 \begin{itemize}
 \item[(i)] For $n \in \N_0$, $\theta(h_n^\alpha)= \Delta^{-\alpha} \mathcal{T}(n)$, in particular $\theta (e_0)=I$ and $\theta(e_1)=T$.
 \item[(ii)] For  $f\in \tau^{\alpha}(\phi)$ such that $\Delta f\in \tau^{\alpha}(\phi)$ and $x\in X$,   $ T \theta(\Delta f)x= (I-T)\theta(f)x-f(0)x.$
 \item[(iii)] In the case that $\displaystyle\sup_{n\in \N_0}{(k^{\beta-\alpha}\ast \phi)(n)\over  \psi(n)}<\infty$, for $0<\alpha<\beta$ and $\psi\in\omega_{\beta, loc}$, then $\tau^{\beta}(\psi)\hookrightarrow \tau^{\alpha}(\phi)$ and
 $$
 \theta(f)x= \displaystyle\sum_{n=0}^{\infty}W^{\beta}f(n)\Delta^{-\beta} \mathcal{T}(n)x,\qquad x\in X,\quad f\in \tau^{\beta}(\psi).
 $$
 \item[(iv)] If $\Vert T\Vert \le a$ for some $a>0$, then
 $ \theta(f)x=\sum_{n=0}^{\infty}f(n)T^n (x),$ for $f\in  \tau^{\alpha}(\phi)\cap \ell^1_{a^n}, $
in particular $\theta(p_\lambda)= (\lambda-T)^{-1}$ for $\vert \lambda\vert>a$.

 \end{itemize}
\end{theorem}
\begin{proof} Note that the map $\theta$ is well-defined, lineal and continuous, $ \Vert \theta(f)x\rVert\leq C q_{\alpha}(f)\lVert x\rVert,$ for $f\in \tau^{\alpha}(\phi)$ and $x\in X.$ To see that $\theta$ is algebra homomorphism is sufficient to prove that $\theta(f*g)=\theta(f)\theta(g)$ for $f,g\in c_{0,0}.$ By Lemma \ref{LemmaTech},
 we get that\begin{eqnarray*}
\theta(f*g)x&=&\displaystyle\sum_{n=0}^{\infty}W^{\alpha}(f*g)(n)\Delta^{-\alpha}\mathcal{T}(n)x \\ &= & \displaystyle\sum_{n=0}^{\infty}\displaystyle\sum_{j=0}^n W^{\alpha}g(j)\displaystyle\sum_{p=n-j}^n k^{\alpha}(p-n+j)W^{\alpha}f(p) \Delta^{-\alpha}\mathcal{T}(n)x \\
&&-\displaystyle\sum_{n=0}^{\infty}\displaystyle\sum_{j=n+1}^{\infty} W^{\alpha}g(j)\displaystyle\sum_{p=n+1}^{\infty} k^{\alpha}(p-n+j)W^{\alpha}f(p)\Delta^{-\alpha}\mathcal{T}(n)x.
\end{eqnarray*}
We apply Fubini theorem to get that \begin{eqnarray*}
\theta(f*g)x&=&\displaystyle\sum_{j=0}^{\infty}W^{\alpha}g(j)\displaystyle\sum_{p=0}^j W^{\alpha}f(p)\displaystyle\sum_{n=j}^{p+j} k^{\alpha}(p-n+j)\Delta^{-\alpha}\mathcal{T}(n)x \\
&&+\displaystyle\sum_{j=0}^{\infty}W^{\alpha}g(j)\displaystyle\sum_{p=j+1}^{\infty} W^{\alpha}f(p)\displaystyle\sum_{n=p}^{p+j} k^{\alpha}(p-n+j)\Delta^{-\alpha}\mathcal{T}(n)x \\
&&-\displaystyle\sum_{j=1}^{\infty}W^{\alpha}g(j)\displaystyle\sum_{p=1}^j W^{\alpha}f(p)\displaystyle\sum_{n=0}^{p-1} k^{\alpha}(p-n+j)\Delta^{-\alpha}\mathcal{T}(n)x \\
&&-\displaystyle\sum_{j=1}^{\infty}W^{\alpha}g(j)\displaystyle\sum_{p=j+1}^{\infty} W^{\alpha}f(p)\displaystyle\sum_{n=0}^{j-1} k^{\alpha}(p-n+j)\Delta^{-\alpha}\mathcal{T}(n)x \\
&=&\displaystyle\sum_{j=1}^{\infty}W^{\alpha}g(j)\displaystyle\sum_{p=1}^j W^{\alpha}f(p)\biggl(\displaystyle\sum_{n=j}^{p+j}-\displaystyle\sum_{n=0}^{p-1}\biggr) k^{\alpha}(p-n+j)\Delta^{-\alpha}\mathcal{T}(n)x+W^{\alpha}g(0)W^{\alpha}f(0)x \\
&&+\displaystyle\sum_{j=0}^{\infty}W^{\alpha}g(j)\displaystyle\sum_{p=j+1}^{\infty} W^{\alpha}f(p)\biggl(\displaystyle\sum_{n=p}^{p+j}-\displaystyle\sum_{n=0}^{j-1}\biggr) k^{\alpha}(p-n+j)\Delta^{-\alpha}\mathcal{T}(n)x \\
&=&\displaystyle\sum_{j=1}^{\infty}W^{\alpha}g(j)\displaystyle\sum_{p=1}^j W^{\alpha}f(p)\Delta^{-\alpha}\mathcal{T}(p)\Delta^{-\alpha}\mathcal{T}(j)x+W^{\alpha}g(0)W^{\alpha}f(0)x \\
&&+\displaystyle\sum_{j=0}^{\infty}W^{\alpha}g(j)\displaystyle\sum_{p=j+1}^{\infty} W^{\alpha}f(p)\Delta^{-\alpha}\mathcal{T}(p)\Delta^{-\alpha}\mathcal{T}(j)x =\theta(f)\theta(g)x.
\end{eqnarray*}
where we have used the identity (\ref{eq4.2}).

(i) Note that $W^\alpha h^\alpha_n=e_n$, see Example \ref{ex2.3} (ii), and then $\theta(h_n^\alpha)= \Delta^{-\alpha} \mathcal{T}(n)$ for $n\in \N_0$. As $e_0=h_0$ and $e_1=h_1^\alpha-\alpha h_0^\alpha$,   it is clear that $\theta(e_0)=I$ and $\theta(e_1)=  T$. (ii) Now, for $f \in \tau^\alpha(\phi)$ such that $\Delta f\in \tau^\alpha(\phi) $ and $x\in X$, we have that \begin{eqnarray*}
T\theta(\Delta f)x&=& T\biggl(\displaystyle\sum_{n=0}^{\infty}W^{\alpha}f(n+1)\Delta^{-\alpha}\mathcal{T}(n)x- \displaystyle\sum_{n=0}^{\infty}W^{\alpha}f(n)\Delta^{-\alpha}\mathcal{T}(n)x\biggr)\\
&=&\displaystyle\sum_{n=0}^{\infty}W^{\alpha}f(n+1)\left(\Delta^{-\alpha}\mathcal{T}(n+1)x-k^{\alpha}(n+1)x\right)-T\displaystyle\sum_{n=0}^{\infty}W^{\alpha}f(n)\Delta^{-\alpha}\mathcal{T}(n)x \\
&=&(I-T)\theta(f)x-W^{\alpha}f(0)\Delta^{-\alpha}\mathcal{T}(0)x-\displaystyle\sum_{n=0}^{\infty}W^{\alpha}f(n+1)k^{\alpha}(n+1)x \\
&=&(I-T)\theta(f)x-\displaystyle\sum_{n=0}^{\infty}W^{\alpha}f(n)k^{\alpha}(n)x=(I-T)\theta(f)x-f(0)x,
\end{eqnarray*}
where we have applied that $
T\Delta^{-\alpha}\mathcal{T}(n)
=\Delta^{-\alpha}\mathcal{T}(n+1)-k^{\alpha}(n+1)$ and
 $\displaystyle\sum_{n=0}^{\infty}W^{\alpha}f(n)k^{\alpha}(n)=f(0)$ for $f\in\tau(\phi).$ (iii)  Suppose that $\displaystyle\sup_{n\in \N_0}{(k^{\beta-\alpha}\ast \phi)(n)\over  \psi(n)}<\infty$, with $0<\alpha<\beta$ and $\psi\in\omega_{\beta, loc}$, then  it is straightforward to check that $\tau^{\beta}(\psi)\hookrightarrow \tau^{\alpha}(\phi)$ and $$\displaystyle\sum_{n=0}^{\infty}W^{\alpha}f(n)\Delta^{-\alpha}\mathcal{T}(n)x=\displaystyle\sum_{n=0}^{\infty}W^{\beta}f(n)\Delta^{-\beta}\mathcal{T}(n)x,\qquad f\in \tau^{\beta}(\psi),\, x\in X,$$
where we have applied Proposition \ref{SemPro} (v) and Remark \ref{Duality}. (iv) Now take $a>0$ such that $\Vert T\Vert \le a$ and then $\sigma(T)\subset \{ z\in \C\,\,\vert\,\,\vert z\vert \le a\}$. For $f\in  \tau^{\alpha}(\phi)\cap \ell^1_{a^n}, $ we apply Remark  \ref{Duality} to get $$ \theta(f)x=\sum_{n=0}^{\infty}f(n)T^n (x), \qquad x\in X.
$$
In particular $p_\lambda \in \tau^{\alpha}(\phi)\cap \ell({a^n})$ for $\vert \lambda \vert >a$ and  $\theta(p_\lambda)x=\displaystyle{{1\over \lambda}\sum_{n=0}^\infty{T^n\over \lambda^n}x}=(\lambda-T)^{-1}x$ for $x\in X$.
\end{proof}

\begin{corollary}\label{reci} Let $\alpha>0$, $\phi \in \omega_\alpha$ and $\theta: \tau^{\alpha}(\phi)\to {\mathcal B}(X)$ be an algebra homomorphism. Then there exists $T\in {\mathcal B}(X)$ such that
$$
\theta(f)x=\sum_{n=0}^{\infty}W^{\alpha}f(n)\Delta^{-\alpha}\mathcal{T}(n)x, \qquad f\in \tau^{\alpha}(\phi), \quad x\in X;
$$
in particular $\theta(h_n^\alpha)=\Delta^{-\alpha}\mathcal{T}(n)$ for $n\in \N_0$ and $\theta(p_\lambda)=(\lambda-T)^{-1}$ for $\vert \lambda\vert > \Vert T\Vert.$
\end{corollary}

\begin{proof} Take $T:=\theta(e_1)$. Note that $e_1=h_1^\alpha-\alpha h_0^\alpha$, see Example \ref{ex2.3} (ii), and
$h_n^\alpha =\Delta^{-\alpha} \mathcal{E}(n)$ for $n\in \N_0$ where $ \mathcal{E}(n)=e^{\ast n}_1$, see Example \ref{canoni}.
By Example \ref{NormEquiv}, $f=\displaystyle\sum_{j=0}^\infty W^\alpha f(n) h_n^\alpha$ for $f\in  \tau^{\alpha}(\phi)$, we apply the continuity of $\theta$ to get
\begin{eqnarray*}
\theta(h_n^\alpha)x&=& \sum_{j=0}^nk^\alpha(n-j)\left(\theta(e_1)\right)^jx=\Delta^{-\alpha}\mathcal{T}(n)x;\cr
\theta(f)x&=&\sum_{n=0}^\infty W^\alpha f(n)\theta( h_n^\alpha)x=\sum_{n=0}^{\infty}W^{\alpha}f(n)\Delta^{-\alpha}\mathcal{T}(n)x,
\end{eqnarray*}
for $x\in X$. By Theorem \ref{homomorphism} (iv), we conclude the proof.
\end{proof}

By Theorem \ref {homomorphism} and Corollary \ref{reci}, we obtain the following characterizations of  $(C, \alpha)$-bounded  and power-bounded operators.

\begin{corollary}\label{cor5.7} Let $T \in \mathcal{B}(X)$  and $\alpha > 0$ be given. The following assertions are equivalent:
\begin{itemize}
\item[(i)] $T$ is $(C, \alpha)$-bounded operator.
\item[(ii)] There exists a bounded algebra homomorphism $\theta : \tau^{\alpha}(k^{\alpha+1}) \to \mathcal{B}(X)$ such that $\theta(e_1)=T.$
\end{itemize}
In the limit case, the following assertions are equivalent:
\begin{itemize}
\item[(a)] $T$ is power bounded.
\item[(b)] There exists a bounded algebra homomorphism $\theta : \ell^1 \to \mathcal{B}(X)$ such that $\theta(e_1)= T.$
    \item[(c)] For any $0<\alpha<1$, there exist bounded algebra homomorphisms $\theta_\alpha : \tau^{\alpha}(k^{\alpha+1}) \to \mathcal{B}(X)$ such that $\theta_\alpha(e_1)=T$ and
        $\displaystyle{
        \sup_{0<\alpha<1}\Vert \theta_\alpha\Vert <\infty.}
        $
\end{itemize}
\end{corollary}

\begin{proof} Due to previous results, we only have to check that (c) implies (b).  Since the map $\theta_\alpha$ is an algebra homomorphism  then $\theta_\alpha(e_n)= T^n$, $\theta_\alpha(f)$ is well defined for $f\in c_{0,0}$ and is independent of $\alpha$. Take $C>0$ such that $\sup_{0<\alpha<1}\Vert \theta_\alpha\Vert<C$.  We define $\theta(f):=\theta_\alpha(f)$ for $ f\in c_{0,0}$ and some given $\alpha \in (0,1)$.
Then $
\Vert \theta(f)\Vert=
\Vert \theta_\alpha(f)\Vert \le C \, q_\alpha(f) $ for $f\in c_{0,0}.$
 By Theorem  \ref{th3.1} (ii), we get that
$
\Vert \theta(f)\Vert \le C \Vert f\Vert_1,$  for $f\in c_{0,0}$ and we conclude the result by density.
\end{proof}

\section{The $Z$-transform and resolvent operators}
\setcounter{theorem}{0}
\setcounter{equation}{0}

Let $f:\N_0\to X$ be a scalar sequence on a Banach space $X$.  We also recall that the $Z$-transform of a given sequence $f:\mathbb{N}_0 \to X$ is defined by
\begin{equation} \label{zeta}
\tilde f(z)= \sum_{n=0}^{\infty} f(n) z^{-n},
\end{equation}
for all $z$ such that this series converges. The set of numbers $z$ in the complex plane for which series \eqref{zeta} converges
is called the region of convergence of $\tilde f$. The uniqueness of the inverse $Z$-transform may be established as follows: suppose that there are two sequences $f$, and $g$ with the same $Z$-transform,
that is,
$$ \sum_{n=0}^{\infty} f(n) z^{-n}= \sum_{n=0}^{\infty}g(n) z^{-n}, \qquad  \vert z\vert > R.$$
It follows from Laurent's theorem that $f(n)=g(n)$ for $n\in \N_0$.

Let $\phi:\N_0\to (0,\infty)$ be a  sequence such that $\phi(n)\le Ca^n$ for some $C>0$ and $a> 0$. To follow the notation given in \cite{cho}, we write $\omega=\log(a)$ and $\omega$ is a bound for the counting measure supported on $\N_0$, i.e., $\epsilon_\lambda \in \ell^1_\phi$ for $\lambda >\omega$ where $\epsilon_{\lambda}(n):=e^{-\lambda n}$ and $n\in \N_0$. Let $C^\infty((\omega,\infty),X)$ be the space of
$X$-valued functions on $(\omega, \infty)$ infinitely differentiable in the norm topology of $X$. For $r\in C^\infty((\omega,\infty),X)$, set
$$\Vert r\Vert_{W, \phi, \omega}:=\sup\{{\Vert r(\lambda)\Vert \over \Vert \beta_{k, \lambda}\Vert_{1, \phi}}\,\, \vert \,\,k\in \N_0,\lambda>\omega\},
$$
where $ \beta_{k, \lambda}(n)=n^k e^{-\lambda n}$ for $n\in \N_0$ and $\lambda >\omega$. The Widder space $C^\infty_W((\omega,\infty),X; \phi) $ is defined by
$$
C^\infty_W((\omega,\infty),X; \phi)=\{r \in C^\infty((\omega,\infty),X)\, \, \vert \, \, \Vert r\Vert_{W, \phi, \omega}<\infty\}.
$$
Endowed with the norm  $\Vert \cdot \Vert_{W, \phi, \omega}$, the space $C^\infty_W((\omega,\infty),X;\phi)$ is a Banach space, see more details in \cite[Section 1]{cho}. A direct consequence of \cite[Theorem 1.2]{cho} is the following result.

\begin{theorem}\label{widder} Let $\phi:\N_0\to (0,\infty)$ be a sequence such that $\phi(n)\le Ca^n$ for some $C>0$ and $a> 0$. Take now $f:\N_0\to X$ a vector-valued sequence. Then the following assertions are equivalent.

\begin{itemize}
\item[(i)] $\displaystyle{\sup_{n\in \N_0}{\Vert f(n)\Vert \over \phi(n)}}<\infty.$
\item[(ii)] There exists $\theta: \ell^1_\phi\to X$ such that $\theta(\lambda p_\lambda)= \tilde f(\lambda)$ for $ \lambda > a$.
\item[(iii)] $\tilde f\circ \exp\in C^\infty_W((\log(a),\infty),X; \phi).$

\end{itemize}
\begin{proof} To show that (i) implies (ii), we define $\theta( g):=\sum_{n=0}^\infty g(n)f(n)$ for $g=(g(n))_{n\ge 0}\in \ell^1_\phi$. Now consider the part (ii). We define $h(n):=\theta(e_n)$ for $n\in \N_0$.
It is clear that $\displaystyle{\sup_{n\in \N_0}{\Vert h(n)\Vert \over \phi(n)}}<\infty$ and
$$
\tilde f(\lambda)=\theta(\lambda p_\lambda)=\sum_{n\in \N_0}\theta(e_n)\lambda^n=\tilde h(\lambda), \qquad \vert \lambda\vert > a,
$$
where we conclude that $h(n)=f(n)$ for $n\in \N_0$ and part (i) is proved. Now take again part (ii).
Due to  \cite[Theorem 1.2]{cho},
$$
\theta(\epsilon_\mu)=\theta(\exp(\mu) p_{\exp(\mu)})= (\tilde f\circ \exp)(\mu), \qquad \mu >\log(a),
$$
and we conclude the part (iii). Finally  suppose that $\tilde f\circ \exp\in C^\infty_W((\log(a),\infty),X; \phi).$ Again by \cite[Theorem 1.2]{cho}, there exists a bounded homomorphism  $\theta: \ell^1_\phi\to X$ such that $\theta(\epsilon_\mu)= (\tilde f\circ \exp)(\mu)$ for $ \mu > \log(a)$. Since $\epsilon_\mu(n)=e^{-\mu n}=e^{\mu}p_{e^\mu}(n)$, we conclude that  $\theta(\lambda p_\lambda)= \tilde f(\lambda)$ for $ \lambda > a$.
\end{proof}
\end{theorem}

\begin{remark}{\rm  Note that  Theorem \ref{widder} is closely connected to \cite[Theorem 4.2]{cho}, where the Banach space $X$ has the Radon-Nikodym property, RNP, to may identity  the Widder space $C^\infty_W((\omega,\infty),X; \hbox{m})$ and $L^\infty(\R_+, X; \hbox{m})$. The RNP is a  well-known property in the theory of function spaces. This property passes to closed subspaces (hereditary property) and is enjoyed by any reflexive space, any separable dual space, and any $\ell^1(\Gamma)$ space, where $\Gamma$ is a set, see definitions and more details in \cite[Section 1.2]{ABHN}.}
\end{remark}

In the well-known scalar version, $X=\C$,   the following $Z$-transforms are obtained directly:
\begin{eqnarray*}
\widetilde{e_n}(z)&=& z^{-n}, \qquad z\not=0, \quad n\in \N_0;\\
\widetilde{ k^{\alpha}}(z)&=& \displaystyle{z^\alpha \over (z-1)^{\alpha}}, \qquad \vert z\vert >1;\\
\widetilde{p_\lambda}(z)&=&\displaystyle{z\over z\lambda-1}, \qquad \vert z\vert >{1\over \vert \lambda\vert}, \,\, \lambda\in \C\backslash\{0\}, \\
\widetilde{h_n^{\alpha}}(z)&=&\displaystyle\sum_{j=0}^{n}k^{\alpha}(n-j)z^{-j},\qquad z\neq 0.
\end{eqnarray*}
It is also well-known that \begin{equation}\label{convo}
\widetilde{(f\ast g)}(z)= \widetilde{f}(z)\widetilde {g}(z), \end{equation} wherever these $Z$-transforms converge on $z\in \C$,
see these results and many other properties of the $Z$-transform in, for example \cite[Chapter 6]{Elaydi}. In particular, given $\alpha>0$ and $f:\N_0\to X$ such that $\tilde f(z)$ converges for $\vert z\vert >R$, then
$$
\widetilde{(\Delta^{-\alpha} f)}(z)=  \displaystyle{z^\alpha \over (z-1)^{\alpha}}\widetilde{f}(z), \qquad \vert z\vert >\max\{R, 1\}.
$$

\bigskip

We denote by $\,_nf(m):=f(n+m)$ for all $m,n\in\N_0.$ Next technical lemma for the $Z$-transform is applied in  Theorem \ref{th5.5}. Similar results hold for the Laplace transform, see for example \cite[Proposition 4.1]{Ke-Li-Mi}.

\begin{lemma}\label{rt} Let $X$ be a Banach space,  $f:\N_0\to \C$ a scalar sequence and $S:\N_0\to \mathcal{B}(X)$ a vector-operator valued sequence. Then \begin{eqnarray*}
\frac{1}{\mu-\lambda}\widetilde{f}(\mu)\biggl( \mu\widetilde{S}(\lambda)x-\lambda\widetilde{S}(\mu)x \biggr)&=&\displaystyle\sum_{n=0}^{\infty}\lambda^{-n}\sum_{m=0}^{\infty}\mu^{-m}(f*\,_nS)(m)x,\\
\frac{1}{\mu-\lambda}\biggl( \mu\widetilde{f}(\lambda)-\lambda\widetilde{f}(\mu) \biggr)\widetilde{S}(\mu)x&=&\displaystyle\sum_{n=0}^{\infty}\lambda^{-n}\sum_{m=0}^{\infty}\mu^{-m}(\,_nf*S)(m)x,
\end{eqnarray*}
for $|\lambda|>|\mu|$ sufficiently large where the double $Z$-transform converge and $x\in X$.
\end{lemma}
\begin{proof}
To show the first equality, note that, $$
\displaystyle\widetilde{\,_nS}(\mu)x=\sum_{m=0}^{\infty}\mu^{-m}S(m+n)x=\mu^{n}\sum_{j=n}^{\infty}\mu^{-j}S(j)x=\mu^{n}\biggl(\widetilde{S}(\mu)x-\sum_{j=0}^{n-1}\mu^{-j}S(j)x\biggr),$$ for $x\in X$ and $n\geq 1$. Then we get that
\begin{displaymath}\begin{array}{l}
\displaystyle\sum_{n=0}^{\infty}\lambda^{-n}\sum_{m=0}^{\infty}\mu^{-m}(f*\,_nS)(m)x=\widetilde{f}(\mu)\displaystyle\sum_{n=0}^{\infty}\lambda^{-n}\widetilde{\,_nS}(\mu)x=\widetilde{f}(\mu)\biggl( \widetilde{S}(\mu)x+\displaystyle\sum_{n=1}^{\infty}\lambda^{-n}\widetilde{\,_nS}(\mu)x\biggr)\\
=\displaystyle\widetilde{f}(\mu)\widetilde{S}(\mu)x\sum_{n=0}^{\infty}\biggl(\frac{\mu}{\lambda}\biggr)^n-\widetilde{f}(\mu)\displaystyle\sum_{n=1}^{\infty}\biggl(\frac{\mu}{\lambda}\biggr)^n\sum_{j=0}^{n-1}\mu^{-j}S(j)x.
\end{array}\end{displaymath} where we have applied the equality (\ref{convo}). Finally, as $$\displaystyle\sum_{n=1}^{\infty}\biggl(\frac{\mu}{\lambda}\biggr)^n\sum_{j=0}^{n-1}\mu^{-j}S(j)x=\sum_{j=0}^{\infty}\mu^{-j}S(j)x\displaystyle\sum_{n=j+1}^{\infty}\biggl(\frac{\mu}{\lambda}\biggr)^n=\frac{\mu}{\lambda-\mu}\widetilde{S}(\lambda)x,$$ we conclude that $$\displaystyle\sum_{n=0}^{\infty}\lambda^{-n}\sum_{m=0}^{\infty}\mu^{-m}(f*\,_n S)(m)x=\frac{1}{\lambda-\mu}\widetilde{f}(\mu)\biggl( \lambda\widetilde{S}(\mu)x-\mu\widetilde{S}(\lambda)x \biggr),$$
for $|\lambda|>|\mu|$ sufficiently large  and $x\in X$. Following these ideas, the second equality  is also shown.
\end{proof}

\begin{theorem}\label{th5.5} Let $\alpha\geq 0$, $X$ a Banach space,  $\{T_n\}_{n\in\N_0}\subset \mathcal{B}(X)$ such that $T_0=I$, $\Vert T_n\Vert \le C \phi(n) \le C 'a^n$ ($\phi\in \omega_{\alpha} $ and
$a>1$) for all $n\in\N_0$ with $C,C'>0$.
 The following statements are equivalent: \begin{itemize}
\item[(i)] The operator-valued sequence $\{T_n\}_{n\in\N_0}$ satisfies the equation \eqref{eq4.2}.
\item[(ii)] There exists  a bounded algebra homomorphism $\theta: \tau^{\alpha}(\phi)\to  \mathcal{B}(X)$ such that $\theta(h_n^\alpha)=T_n$ for $n\in \N_0$.
\item[(iii)] The family $\{ R(\lambda) \}_{|\lambda|>a}$ defined by
$$\displaystyle R(\lambda)x:=\frac{(\lambda-1)^{\alpha}}{\lambda^{\alpha+1}}\displaystyle\sum_{n=0}^{\infty}\lambda^{-n}T_n(x),\qquad |\lambda|>a,\, x\in X,$$  is a pseudo-resolvent.
\end{itemize}
In these cases the generator of  $\{T_n\}_{n\in\N_0}$, defined by $T:=T_1-\alpha I$ in  Remark \ref{generador}, satisfies that $T_n = \Delta^{-\alpha}\mathcal{T}(n)$ for $n\in \N_0$,  $\theta(e_1)=T$, $\{\lambda \in \C\,\,\vert \, \vert \lambda\vert>a\}\subset \rho(T)$ and
$$R(\lambda)=(\lambda -T)^{-1}, \qquad \vert\lambda\vert >a.
$$

\end{theorem}
\begin{proof}
The proof (i)$\Rightarrow$(ii) is a direct consequence of Theorem \ref{TheoremEcFunc} and Theorem \ref{homomorphism}. To show that (ii)$\Rightarrow$(iii), we use that Corollary \ref{reci}. Finally we prove (iii)$\Rightarrow$(i). It is clear that $$R(\lambda)=\frac{\widetilde{{\frak T}}(\lambda)}{\lambda\widetilde{k^{\alpha}}(\lambda)}, \qquad  \vert\lambda\vert >a,$$ where ${\frak T}=\{T_n\}_{n\in\N_0}$ and $\widetilde{{\frak T}}$  is given by (\ref{zeta}).  Since $\{ R(\lambda) \}_{|\lambda|>a}$ is a pseudo-resolvent, then $$(\mu-\lambda)\frac{\widetilde{{\frak T}}(\lambda)\widetilde{{\frak T}}(\mu)}{\lambda\widetilde{k^{\alpha}}(\lambda)\mu\widetilde{k^{\alpha}}(\mu)}=\frac{\widetilde{{\frak T}}(\lambda)}{\lambda\widetilde{k^{\alpha}}(\lambda)}-\frac{\widetilde{{\frak T}}(\mu)}{\mu\widetilde{k^{\alpha}}(\mu)},\qquad \vert\lambda\vert, \vert \mu\vert >a, \quad \mu\not=\lambda,$$ so $$\widetilde{{\frak T}}(\lambda)\widetilde{{\frak T}}(\mu)=\frac{1}{\mu-\lambda}\biggl( \mu\widetilde{k^{\alpha}}(\mu)\widetilde{{\frak T}}(\lambda)-\lambda\widetilde{k^{\alpha}}(\lambda)\widetilde{{\frak T}}(\mu) \biggr), \qquad \vert\lambda\vert, \vert \mu\vert >a, \quad \mu\not=\lambda.$$ On the other hand, note that the condition (\ref{eq4.2}) is expressed by \begin{displaymath}\begin{array}{l}
\displaystyle(k^{\alpha}*\,_n{\frak T})(m)-(\,_nk^{\alpha}*{\frak T})(m)+k^{\alpha}(n)T_m =\sum_{u=n}^{n+m}k^{\alpha}(n+m-u)T_u-\sum_{u=0}^{m-1}k^{\alpha}(n+m-u)T_u,\end{array}\end{displaymath} for $m\geq 1$ and $n\geq 0.$  We apply Lemma \ref{rt}  and do some simple operations to get that
$$
\sum_{n=0}^{\infty}\lambda^{-n}\sum_{m=0}^{\infty}\mu^{-m}\left(\displaystyle(k^{\alpha}*\,_n{\frak T})(m)-(\,_nk^{\alpha}*{\frak T})(m)+k^{\alpha}(n)T_m \right)={ \mu\widetilde{k^{\alpha}}(\mu)\widetilde{{\frak T}}(\lambda)-\lambda\widetilde{k^{\alpha}}(\lambda)\widetilde{{\frak T}}(\mu)\over\mu-\lambda},
 $$ for $\vert\lambda\vert, \vert \mu\vert >a, $ and $ \mu\not=\lambda$. Then we conclude that $\{T_n\}_{n\in\N_0}$ satisfies \eqref{eq4.2}, as consequence of the injectivity of the double $Z$-transform.
 Finally, by Corollary \ref{reci}
 $$
 R(\lambda)=\theta(p_\lambda)=(\lambda -T)^{-1}, \qquad \vert\lambda\vert >a,
 $$
and we finish the proof.
\end{proof}

\section{Applications, examples and final comments}
\setcounter{theorem}{0}
\setcounter{equation}{0}

In this last section, we present some applications, comments, examples and counterexamples of some results presented in this paper.

\subsection{Bounds for Abel means} Given $T\in {\mathcal B}(X)$ and $0\leq r<1$ the Abel mean of order $r$ of operator $T$, $A_r(T),$ is defined by
$$
A_r(T)x:= (1-r)\sum_{n=0}^\infty r^nT^n(x), \qquad x\in X,
$$
when this series converges for some $r\in[0,1)$, see for example \cite{LSS}. Note that for $0< r<\frac{1}{r(T)}$ then ${1\over r}\in\rho(T) $ and
$$
A_r(T)=\frac{(1-r)}{r}(\frac{1}{r}-T)^{-1}, \qquad  0< r <\min\{1, {1\over r(T)}\},
$$
where $r(T)=\lim_{n\to\infty}\lVert T^{n}\rVert^{\frac{1}{n}}$ denotes the spectral radius of $T.$

The next theorem improves  \cite[Proposition 2.1 (i)]{LSS} given for $\alpha \in \{0,1\}$.
\begin{theorem}\label{abels} Take $\alpha\ge 0 $ and $T\in\mathcal{B}(X)$. Then
$$
A_r(T)x=(1-r)^{\alpha+1}\sum_{n=0}^\infty r^{n}\Delta^{-\alpha}\mathcal{T}(n)x, \qquad 0\leq r <\min\{1, {1\over r(T)}\}.
$$
In the case that $\Vert\Delta^{-\alpha}\mathcal{T}(n)\Vert\le Ck^{\gamma+1}(n)$ for $n\ge 1$ and $\gamma\ge \alpha$ then
$$
\Vert A_r(T)\Vert\le  C(1-r)^{-(\gamma-\alpha)}, \qquad 0\leq r<1.
$$
In particular if $T$ is a $(C, \alpha)$-bounded operator, then $ \sup_{0\leq r<1}\Vert  A_r(T)\Vert <\infty$.
\end{theorem}

\begin{proof} Let $\alpha\ge 0$, and $p_{1\over r}(n)=r^{n+1}$ for $0<r<1$. By Remark \ref{Duality}, we have that

\begin{eqnarray*}
A_r(T)x &=& (1-r)\sum_{n=0}^\infty r^nT^n(x)= \frac{1-r}{r}\sum_{n=0}^\infty W^\alpha p_{1\over r}(n)\Delta^{-\alpha}\mathcal{T}(n)x\\
&=& \frac{(1-r)^{\alpha+1}}{r}\sum_{n=0}^\infty p_{1\over r}(n)\Delta^{-\alpha}\mathcal{T}(n)x
= (1-r)^{\alpha+1}\sum_{n=0}^\infty r^{n}\Delta^{-\alpha}\mathcal{T}(n)x,
\end{eqnarray*}
 where we have used Example \ref{ex2.3} (i) for $0<r<\min\{1, {1\over r(T)}\}$. For $r=0$ is obvious.

In the case that $\Vert\Delta^{-\alpha}\mathcal{T}(n)\Vert\le Ck^{\gamma+1}(n)$ for $n\ge 1$ and $\gamma\ge \alpha$, there exists a bounded algebra homomorphism $\theta:\tau^{\alpha}(k^{\gamma+1})\to \mathcal{B}(X)$ by Theorem \ref{homomorphism}. Note that $p_{1\over r}\in  \tau^{\alpha}(k^{\gamma+1})$ and
${A_r(T)= \frac{1-r}{r}\theta( p_{1\over r}),}$ for  $0<r<1.$
By formula (\ref{normass}), we obtain that
$$
\Vert A_r(T)\Vert \le C \frac{1-r}{r}q_{k^{\gamma+1}}(p_{1\over r})=  C\frac{1-r}{r}{r\over (1-r)^{\gamma+1-\alpha}}={C\over(1-r)^{\gamma-\alpha}}, \qquad 0<r<1,
$$
and we conclude the proof.
\end{proof}

\begin{remark}{\rm  If we consider $\lVert T^n \rVert\leq C n^{\gamma},$ with $\gamma\geq 0,$ using that $n^{\gamma}\leq \Gamma(\gamma+1)k^{\gamma+1}(n)$ which follows easily from \eqref{gau}, we get that  $$\lVert A_r(T)\rVert\leq  C\Gamma(\gamma+1)(1-r)^{-\gamma},$$ which improves the bound of \cite[Proposition 2.1 (i) (2.3)]{LSS}. Use similar arguments to improve the bound of \cite[Proposition 2.1 (i) (2.4)]{LSS}.}
\end{remark}

\begin{remark}{\rm An inverse result exists on Banach lattices, see \cite[Corollary 3.2]{LSS}, which proves that for any $\alpha>-1$ and a positive bounded operator $T,$  $\{ (1-r)^{\alpha}A_r(T),\ 0\leq r<1\}$ is bounded if only if $\lVert \Delta^{-1}\mathcal{T}(n)\rVert \leq C (n+1)^{\alpha},$ $n\in\N_0.$ In particular,  $T$ is Abel-mean bounded if only if  is $(C,1)$-bounded. Note that there are examples of positive $(C,1)$-bounded operators in Banach lattices which are not power bounded, see remarks following \cite[Corollary 3.2]{LSS}.}
\end{remark}

%-------------------------------------------------------------

%-------------------------------------------------------------

\subsection{$\alpha$-Times integrated semigroups and Ces\`{a}ro sums}

Now, let $A$ be a closed linear operator on $X,$ $\alpha> 0$ and $\{S_{\alpha}(t)\}_{t\geq 0}\subset {\mathcal B}(X)$ an $\alpha$-times integrated semigroup generated by $A$, that is,  $S_\alpha(0)=0$, the map $[0,\infty)\to X$, $r\mapsto S_\alpha(r)x$  is strongly continuous and
$$
S_\alpha(t)S_\alpha(s)x={1\over \Gamma(\alpha)}\left(\int_t^{t+s}(t+s-r)^{\alpha-1}S_\alpha(r)xdr- \int_0^{s}(t+s-r)^{\alpha-1} S_\alpha(r)xdr\right), \qquad x\in X,
$$
 for $t,s >0$; for $\alpha=0,$ $\{S_{0}(t)\}_{t\geq 0}$ is an  usual $C_0$-semigroup, $S_0(0)=I$  and $S_0(t+s)=S_0(t)S_0(s)$ for $t,s>0$. In the case that  $\{S_{\alpha}(t)\}_{t\geq 0}$ is a non-degenerate family and $\Vert S_\alpha(t)\Vert \le C e^{\omega t}$ for $C>0$, $\omega \in \R$, then there exists a closed operator, $(A, D(A))$, called the generator of $\{S_{\alpha}(t)\}_{t\geq 0}$, such that
 \begin{equation}\label{resolvent}
 (\lambda-A)^{-1}x=\lambda^\alpha \int_0^\infty e^{-\lambda t} S_\alpha (t)xdt, \qquad \Re \lambda>\omega, \qquad x\in X.
 \end{equation}
 Moreover the following integral equality holds
\begin{equation}\label{integral}
A\int_0^tS_\alpha(s)xds=S_\alpha(t)x-{t^\alpha\over \Gamma(\alpha+1)}x, \qquad t > 0, \quad x\in X.
\end{equation}

 \begin{theorem} \label{resolventinte}Suppose that $\{S_{\alpha}(t)\}_{t\geq 0}$ is an $\alpha$-times integrated semigroup generated by $(A, D(A))$ such that $\Vert S_\alpha(t)\Vert \le C e^{\omega t}$ with $0\le\omega<1$. Then $1\in \rho(A)$, $R:=(1-A)^{-1}$,  ${\mathcal R}(n)=R^n$ and
\begin{eqnarray*}\label{convoresol}
\Delta^{-\alpha}{\mathcal R}(n)x&=&(I-A)\int_0^\infty {e^{-t}t^{n}\over n!}S_\alpha(t)xdt, \qquad n\in \N_0,\cr
&=&\int_0^\infty {e^{-t}t^{n-1}\over (n-1)!}S_\alpha(t)xdt+k^{\alpha+1}(n)x-k^{\alpha+1}(n-1)x, \qquad n\ge 1, \quad x\in X,
\end{eqnarray*}
In particular if $\{S_{\alpha}(t)\}_{t\geq 0}$  has temperated growth, i.e.  $\Vert S_\alpha(t)\Vert \le C t^\alpha$ for $t>0$, then $(I-A)^{-1}$ is a $(C, \alpha)$-bounded operator.
 \end{theorem}
\begin{proof} Take $\lambda$ such that $\lambda\in\rho(A) $ and then
$$
\frac{(-1)^n}{n!}\frac{d^n}{d\lambda^n}(\lambda^{-\alpha}(\lambda-A)^{-1})= \sum_{j=0}^n{k^\alpha(n-j)\over \lambda^{\alpha+n-j}}(\lambda-A)^{-j-1}.
$$
In other hand, for $\lambda $ such that $\Re \lambda >\omega$, we apply formula (\ref{resolvent}) to get that
$$
\frac{(-1)^n}{n!}\frac{d^n}{d\lambda^n}(\lambda^{-\alpha}(\lambda-A)^{-1})x= \int_0^{\infty} {t^n\over n!}e^{-\lambda t}S_{\alpha}(t)x\,dt, \qquad x\in X.
$$
Finally we take $\lambda=1$ and write $R:=(1-A)^{-1}$,  ${\mathcal R}(n)=R^n$ to conclude the first equality for $n\in \N_0$.  Now for $n\ge 1$, we have that
\begin{eqnarray*}
\Delta^{-\alpha}{\mathcal R}(n)x&=& \int_0^\infty {e^{-t}t^{n}\over n!}S_\alpha(t)xdt+A\int_0^\infty {e^{-t}t^{n-1}\over (n-1)!}\left(1-{t\over n}\right)\int_0^tS_\alpha(s)xdsdt\\
&=&\int_0^\infty {e^{-t}t^{n-1}\over (n-1)!}S_\alpha(t)xdt+k^{\alpha+1}(n)x-k^{\alpha+1}(n-1)x, \qquad x\in X,
\end{eqnarray*}
where we have apply the equality (\ref{integral}).

In the case that  $\Vert S_\alpha(t)\Vert \le C t^\alpha$, we use the second equality and that the sequence $k^{\alpha+1}$ is increasing to conclude that
$\displaystyle{\sup_{n\in \N_0}}{\lVert \Delta^{-\alpha}{\mathcal R}(n)\rVert \over k^{\alpha+1}(n)}<\infty$ and $(I-A)^{-1}$ is a $(C, \alpha)$-bounded operator.
\end{proof}

Classical examples of   generators of  temperated  $\alpha$-times integrated semigroup are differential operators $A$ such that their symbol $\hat{A}$ is of the form $\hat{A}=ia$ where $a$ is a real elliptic homogeneous polynomial on $\R^n$ or $a\in C^{\infty}(\R^n\setminus\{0\})$ is a real homogeneous function on $\R^n$ such that if $a(t)=0$ then $t=0,$ see \cite[Theorem 4.2]{Hieber}, and other different examples in \cite[Section 6]{Hieber}.

\begin{remark}{\rm  In the case of uniformly bounded $C_0$-semigroups, i.e. $\{T(t)\}_{t\ge 0}\subset {\mathcal B}(X)$ such that $\sup_{t>0}\Vert T(t)\Vert <\infty$, the resolvent $(1-A)^{-1}$ is power-bounded due to
$$
(1-A)^{-n}x=\int_0^\infty {t^{n-1}\over (n-1)!} e^{-t}T(t)xdt, \qquad x\in X.
$$
Note that Theorem \ref{resolventinte} includes a natural extension of this fact: the resolvent $(1-A)^{-1}$ is a $(C, \alpha)$-bounded operator when $A$ generates a   temperated  $\alpha$-times integrated semigroup.

We may also consider the homomorphism $\theta$ defined in Theorem \ref{homomorphism}, and in this case
$$
\theta(\Delta f)x = -A\theta (f)x - (I-A)f(0)x, \qquad f\in\tau^{\alpha}(k^{\alpha+1}), \quad x\in D(A),
$$ when $A$ generates a  temperated  $\alpha$-times integrated semigroup.}
\end{remark}

\subsection{Counterexamples of bounded homomorphisms}

%--------------------------------------------------------
\begin{example}{\rm In \cite[Section 2]{DL} there is an example of a positive, Ces\`{a}ro bounded  but not
power bounded operator $T$ on the space $\ell^1$. As the author comments in \cite[Section 4. Examples]{De00}, $\Vert  T^n\Vert_1\le K{n/ \ln (n)}$  where $K$ is the uniform bound of the Ces\`{a}ro averages of $T$. In this example $T$ is also a contraction in $\ell^\infty$. In \cite[Section (VI)]{E2}, it is proven that $\sup_{n\ge 0}\Vert T^n\Vert_p\ge (2^k)^{1\over p}$ for any $k\ge 1$ and $1\le p< \infty$. We conclude that $T$ is not power bounded in $\ell^p$ $(1 \le p < \infty)$ and $T$ is a Ces\`{a}ro bounded in $\ell^p$ $(1 \le p \le \infty)$ . By Corollary \ref{cor5.7}, there exists a bounded homomorphism $\theta: \tau^1(k^2)\to {\mathcal B}(\ell^p)$ such that $\theta(e_1)=T$ and extends to $\theta: \ell^1\to {\mathcal B}(\ell^p)$ if and only $p=\infty$. }
\end{example}

\begin{example}{\rm In \cite{To-Ze}, a simple
matrix construction,  which  unifies  different approaches to the
Ritt condition and  ergodicity of matrix semigroups, is studied in detail. Consider the Banach space ${\frak X}:= X\oplus X$ with norm
$$
\Vert x_1\oplus x_2\Vert_{X\oplus X}:= \sqrt{\Vert x_1\Vert^2+\Vert x_2\Vert^2}, \qquad x_1\oplus x_2\in {\frak X}.
$$
 Let the bounded linear operator ${\frak T}$ on ${\frak X}$ be defined by the operator matrix
$$
{\frak T}:=\left(\begin{matrix}T &T-I\\ 0&T \end{matrix}\right)
$$
where $T\in {\mathcal B}(X)$. In \cite[Lemma 2.1]{To-Ze}, some connected properties between $T$ and ${\frak T}$ are given. Now we consider as $X= \ell^2$ and the backward shift operator $T\in {\mathcal L}(\ell^2)$ defined by
$$T((x_n)_{n\ge 0}):=(x_{n+1})_{n\ge 0}, \qquad (x_n)_{n\ge 0}\in \ell^2. $$
By \cite[Example 3.1]{To-Ze}, $\Vert {\frak T}^n\Vert\ge 2n$ and $  {\frak T}
 $ is a $(C, 1)$-bounded operator. We apply Corollary \ref{cor5.7} to conclude that there exists an algebra homomorphisms $\theta: \tau^1(k^2)\to {\mathcal B}({\frak X})$ such $\theta(e_1)={\frak T}$ and it is not extended continuously to $\ell^1$.  In \cite[Remark 3.2]{To-Ze}, the growth    $\Vert {\frak T}^n\Vert\ge 2n$ is pointed at as the fastest possible for a Ces\`{a}ro bounded operator. }
\end{example}

\begin{example}{\rm In \cite[Proposition 4.3]{LSS}, the following example is given.
For any $\gamma$ with $0<\gamma<1$, there exists a positive linear operator $T$ on an $L_1$-space such that
$$\sup_{n\ge0}\Vert\frac{\Delta^{-\gamma}\mathcal{T}(n)}{k^{\gamma+1}(n)}\Vert=\infty, \quad\textrm{but}\quad \sup_{n\ge0}\Vert\frac{\Delta^{-\beta}\mathcal{T}(n)}{k^{\beta+1}(n)}\Vert<\infty \quad\textrm{for all }\beta>\gamma.$$
By  Corollary \ref{cor5.7},  we conclude that there exists a bounded algebra homomorphism $\theta$ such that $\theta:\tau^{\beta}(k^{\beta+1})\to \mathcal{B}(X)$ for all $\beta >\gamma$, $\theta(e_1)=T$, and the homomorphism $\theta$ is not extended continuously to the algebra $\tau^{\gamma}(k^{\gamma+1})$ with $0<\gamma<1$.}
\end{example}

\begin{example}{\rm In \cite[Proposition 4.4 (i)]{LSS}, the following operator is constructed.
Let $dim X=\infty$. For any integer $j\ge0$, there exists a bounded linear operator $T$ on $X$ such that
$$\sup_{n\ge0}\Vert\frac{\Delta^{-(j+1)}\mathcal{T}(n)}{k^{j+2}(n)}\Vert<\infty,  \quad\textrm{but}\quad \sup_{n\ge0}\Vert\frac{\Delta^{-\gamma}\mathcal{T}(n)}{k^{\gamma+1}(n)}\Vert=\infty \quad\textrm{for }0\le\gamma<j+1.$$
By  Corollary \ref{cor5.7},  we conclude that there exists a bounded algebra homomorphism $\theta$ such that $\theta:\tau^{j+1}(k^{j+2})\to \mathcal{B}(X)$, $\theta(e_1)=T$, and the homomorphism $\theta$ is not continuously extended to the algebra $\tau^{\gamma}(k^{\gamma+1})$ with $0\le\gamma<j+1$.}
\end{example}

 \begin{example}{\rm In \cite[Proposition 4.4 (ii)]{LSS}, the following operator is constructed.
Let $dim X=\infty$. There exists a bounded linear operator $T$ on $X$ with $r(T)=1$, $\Vert T\Vert=2,$  and
$$\Vert A_r(T)\Vert\le 1-r,  \quad0<r<1;\quad\textrm{and}\quad \sup_{n\ge0}\Vert\frac{\Delta^{-j}\mathcal{T}(n)}{k^{j+1}(n)}\Vert=\infty, \quad\textrm{for }j\ge 1.$$
Since $k^j(n)\le k^{j+1}(n)$ for $n\ge 0$, we also conclude that $\displaystyle{\Vert\frac{\Delta^{-j}\mathcal{T}(n)}{k^{j}(n)}\Vert=\infty}$ for $j\ge 1$ and the converse of Theorem \ref{abels} does not hold for $\gamma <\alpha$.}
\end{example}

%--------------------------------------------------------

\subsection{Application to Katznelson-Tzafriri theorem}

%--------------------------------------------------------

Let $A(\mathbb{T})$ be the regular convolution Wiener algebra formed by all continuous periodic functions $f(t)=\sum_{n=-\infty}^{\infty}\hat{f}(n) e^{int}, \quad t\in [-\pi,\pi],$ where $(\hat{f}(n))_{n\in\Z}$ are the Fourier coefficients of $f,$ that is $$\hat{f}(n)=\frac{1}{2\pi}\int_{-\pi}^{\pi}f(t)e^{-int}\,dt,\qquad n\in \Z,$$ with the norm $\lVert f \rVert_{A(\mathbb{T})}:=\sum_{n=-\infty}^{\infty}|\hat{f}(n)|,$ and $A_+(\mathbb{T})$ be the convolution closed subalgebra of $A(\mathbb{T})$ where the functions satisfies that $\hat{f}(n)=0$ for $n<0.$ Note that both $A(\mathbb{T})$ and $\ell^1_{\Z},$ and $A_+(\mathbb{T})$ and $\ell^1$ are isometrically isomorphic, where $\ell^1_{\Z}$ denotes the complex summable sequences indexed in $\Z.$

Katznelson and Tzafriri proved in 1986 the following well known  theorem: if $T\in\mathcal{B}(X)$ is power-bounded and $f\in A_+(\mathbb{T})$ is of spectral synthesis in $A(\mathbb{T})$ with respect to $\sigma(T)\cap \mathbb{T},$ then $$\displaystyle\lim_{n\to\infty}\lVert T^n\theta(\hat{f}) \rVert=0,$$ see \cite[Theorem 5]{Katznelson}. Moreover, for $T\in\mathcal{B}(X)$ a power-bounded operator, $\displaystyle\lim_{n\to\infty}\lVert T^n-T^{n+1}\rVert=0$ if and only if $\sigma(T)\cap \mathbb{T}\subseteq\{1\},$ see \cite[Theorem 1]{Katznelson}.

The authors have got some similar results for $(C,\alpha)$-bounded operators, which will appear in a forthcoming paper. We define $A^{\alpha}(\mathbb{T})$ a new regular Wiener algebra contained in $A(\mathbb{T}),$ and $A_+^{\alpha}(\mathbb{T})$ a convolution closed subalgebra of $A^{\alpha}(\mathbb{T}),$ which is isometrically isomorphic to $\tau^{\alpha}(k^{\alpha+1}).$ The result prove that if $\alpha> 0,$ $T\in\mathcal{B}(X)$ is a $(C,\alpha)$-bounded operator and $f\in A_+^{\alpha}(\mathbb{T})$ is of spectral synthesis in $A^{\alpha}(\mathbb{T})$ with respect to $\sigma(T)\cap \mathbb{T},$ then $$\displaystyle\lim_{n\to\infty}\frac{1}{k^{\alpha+1}(n)}\lVert \Delta^{-\alpha} \mathcal{T}(n)\theta(\hat{f}) \rVert=0.$$

On the continuous case, Katznelson-Tzafriri theorems have been proved for $C_0$-semigroups and extended later for $\alpha$-times integrated semigroups, see \cite{Esterle} and \cite{GMM} respectively.

%--------------------------------------------------------

\subsection*{\it Acknowledgments.} This work was done while the second author was on  sabbatical leave, visiting  the University of Zaragoza. He is grateful to the members of the Analysis Group for their kind hospitality.


\begin{thebibliography}{999}

\bibitem{AbdeljaDual} T. Abdeljawad. {\it Dual identities in fractional difference
calculus within Riemann.} Adv. in Diff. Equat., vol.2013,
article 36, 2013.


\bibitem{Abdelja2} T. Abdeljawad and F. M. Atici.  {\it  On the definitions of nabla fractional operators.}. Abstr. Appl. Anal. 2012, Article ID 406757
(2012). Doi:10.1155/2012/406757

%\bibitem{Abdelja} T. Abdeljawad. {\it On Riemann and Caputo fractional diferences.} Comput. Math. Appl. 62 (2011), 1602-1611.

\bibitem{ABHN} W. Arendt, C. J. K. Batty, M. Hieber, F. Neubrander. { \it Vector-valued Laplace transforms and Cauchy
problems,} Second edition, Monographs in Mathematics. {\bf 96},
Birkh\"auser (2011).

\bibitem{AtEl09} F. M. Atici and P. W. Eloe. {\it Initial value problems in discrete fractional calculus.} Proc. Amer. Math. Soc.,
137 (3), (2009), 981-989.

\bibitem{AtSe10} F.M. Atici and S. Seng\"ul. {\it Modeling with fractional difference equations.} J. Math. Anal. Appl., 369 (2010), 1-9.

\bibitem{Lizama1} S. Calzadillas, C. Lizama and J. G. Mesquita. {\it A unified approach to discrete fractional calculus and applications.} Preprint, 2014.

\bibitem{Ch-Mu93} S. Chanillo and B. Muckenhoupt. {\it Weak type estimates for Ces\`{a}ro sums of Jacobi polinomial series.} Mem. Amer. Math. Soc. 102, (1993), vol. 487.


\bibitem{cho} W. Chojnacki. {\it A generalization of the Widder-Arendt theorem.} Proc. of the Edingburgh. Math. Soc., 45 (2002), 161-179.


%\bibitem{Cu-Pa03} E. Cuesta and C. Palencia. {\it  A numerical method for an integro-differential equation with memory in Banach spaces: Qualitative properties.} SIAM J. Numer. Anal., 41 (3) (2003), 1232-1241.

\bibitem{DL} Y. Derriennic and M. Lin, {\it On invariant measures and ergodic theorems for
positive operators. } J. Func. Anal., 13 (1973), 252--267.

\bibitem{De00} Y. Derriennic. {\it On the mean ergodic theorem for Ces\`{a}ro bounded operators.} Colloq. Math. 84/85 (2000), 443-455.

\bibitem{Ed04} E. Ed-Dari. {\it On the $(C,\alpha)$ Ces\`{a}ro bounded operators.} Studia Mathematica 161 (2) (2004), 163-175.

\bibitem{Elaydi} S. Elaydi. {\it An Introduction to Difference Equations.} Undergraduate Texts in Mathematics. Springer. 3rd. Edition, 2005.

\bibitem{E2} R. Emilion, {\it Mean-Bounded operators and
mean ergodic theorems.} J.  Func. Anal., 61 (1985), 1--14.

\bibitem{Esterle} J. Esterle, E. Strouse and F. Zouakia. {\it Stabilit\'e asymptotique de certains semi-groupes d'op\'erateurs et ideaux primaires de $L^1(\R^+)$.} J. Operator Theory, 28 (1992), 203-227.

\bibitem{GM} J. E. Gal\'e and P.J. Miana. {\it One-parameter groups of regular quasimultipliers.} J. Funct. Anal., 237 (2006), 1--53.

\bibitem{GMM} J. E. Gal\'e, M. M. Mart\'inez and P.J. Miana. {\it Katznelson-Tzafriri type theorem for integrated semigroups.} J. Operator Theory, 69 (1) (2013), 59-85.

\bibitem{Gale} J. E. Gal\'e and A. Wawrzy\'nczyk. {\it Standard ideals in weighted algebras of Korenblyum and Wiener types.} Math. Scand., 108 (2011), 291--319.

 \bibitem{Gautschi} W. Gautschi. {\it  Some elementary inequalities relating to the gamma and incomplete gamma function.} J. Math. Phys. 38 (1) (1959), 77--81.


\bibitem{Go12} C.S. Goodrich. {\it On a first-order semipositone discrete fractional boundary value problem.} Arch. Math. 99
(2012), 509-518.

\bibitem{Gr31}  T.H. Gronwall. {\it On the Ces\`{a}ro sums of Fourier's and Laplace's series.} Ann. of Math. 32 (1) (1931),  53--59.

\bibitem{Hieber} M. Hieber. {\it Integrated semigroups and differential operators on $L^p$ spaces.} Math. Ann. 291 (1991), 1-16.

%\bibitem{Hille} E. Hille.  {\it Remarks on ergodic theorems.} Trans. Amer. Math. Soc. 57 (1945), 246-269.

\bibitem{Katznelson} Y. Katznelson and L. Tzafriri. {\it On power bounded operator.} J. Func. Anal., 68 (1986), 313--328.

\bibitem{Ke-Li-Mi} V. Keyantuo, C. Lizama and P.J. Miana. {\it Algebra homomorphisms defined via convoluted
semigroups and cosine functions.} J.  Func. Anal., 257 (2009), 3454--3487.


\bibitem{Ku39} B. Kuttner. {\it Some theorems on Riesz and Ces\`aro sums.} Proc. London. Math. Soc. 45, (1939), 398--409.

\bibitem{LSS} Y.-C. Li, R. Sato and S.-Y. Shaw. {\it Boundednes and growth orders of means of discrete and continuous semigroups of operators.} Studia Mathematica 187 (1) (2008), 1-35.

%\bibitem{Lu86} Ch. Lubich. {\it Discretized fractional calculus.} SIAM J. Math. Anal. 17 (1986) 704--719.



%\bibitem{MeScMo04} M.M. Meerschaert, H.P. Scheffler and J. Mortensen. {\it Vector Gr\"unwald formula for fractional derivatives.} Fractional Calculus Appl. Anal. 7 (1) (2004) 61--81.

%\bibitem{MeTa04} M.M. Meerschaert and C. Tadjeran. {\it  Finite difference approximations for fractional advection dispersion flow equations. } J. Comput. Appl. Math. 127 (2004) 65--77.

%\bibitem{Mo} C. N. Moore.

\bibitem{Su-Ze13} L. Suciu and J. Zem\'anek. {\it Growth conditions on Ces\`aro means of higher order.} Acta Sci. Math. (Szeged) 79 (2013), 545--581.

\bibitem{To-Ze} Y. Tomilov and J. Zem\'anek. {\it A new way of constructing examples in operator ergodic theory.} Math. Proc. Camb. Philos. Soc. 137 (2004), 209--225.

\bibitem{Yo98} T. Yoshimoto. {\it Uniform and strong ergodic theorems in Banach spaces.} Illinois J. Math., 42 (1998), 525-543; Correction, ibid. 43 (1999), 800-801.

\bibitem{Zygmund} A. Zygmund. {\it Trigonometric Series.} 2nd ed. Vols. I, II, Cambridge University
Press, New York, 1959.

\end{thebibliography}
\end{document}